\documentclass[10pt]{amsart}
\usepackage[all]{xy}
\pdfoutput=1
\usepackage{amsmath, amsthm, amssymb,slashed,stmaryrd}
\usepackage{enumitem}
\usepackage{ifpdf}
\usepackage[pdftex]{graphicx}
\usepackage{tikz}
\usetikzlibrary{matrix,arrows,calc}
\usepackage{mathtools}
\usepackage[pdftex,plainpages=false,hypertexnames=false,pdfpagelabels]{hyperref}
\usepackage[capitalise]{cleveref}
\usepackage{amssymb}
\usepackage{tikz-cd}
\usetikzlibrary{arrows,decorations.markings}
\usetikzlibrary{shapes.geometric}

\usepackage{bm}
\usepackage{accents}
\usepackage[super]{nth}
\usepackage{mathrsfs}

\usepackage{appendix}

\numberwithin{equation}{subsection}

\newtheorem{thm}[subsubsection]{Theorem}
\newtheorem{thmalpha}{Theorem}

\newtheorem{lem}[subsubsection]{Lemma}
\newtheorem{prop}[subsubsection]{Proposition}
\newtheorem{fact}[subsubsection]{Fact}
\newtheorem{cor}[subsubsection]{Corollary}
\newtheorem{notation}[subsubsection]{Notation}
\theoremstyle{definition}
\newtheorem{setup}[subsubsection]{Setup}

\newtheorem{defn}[subsubsection]{Definition}
\newtheorem{rmk}[subsubsection]{Remark}
\newtheorem{question}[subsubsection]{Question}

\newtheorem{claim}[subsubsection]{Claim}

\newtheorem*{conjecture*}{Conjecture}
\newtheorem*{theorem*}{Theorem}
\newtheorem*{claim*}{Claim}
\newtheorem*{corollary*}{Corollary}
\newtheorem*{notation*}{Notation}

\DeclareSymbolFont{bbold}{U}{bbold}{m}{n}
\DeclareSymbolFontAlphabet{\mathbbold}{bbold}

\def\Spec{{\rm Spec}}

\def\Gal{{\rm Gal}}

\def\K3{{\rm K3}}

\def\GL{{\rm GL}}

\def\ad{{\rm ad}}

\def\Fqbar{{\overline{\mb{F}}_q}}

\def\Gr{{\rm Gr}}

\newcommand{\Zp}{\mathbb{Z}_p}
\newcommand{\Zl}{\mathbb{Z}_{\ell}}
\newcommand{\Zlbar}{\bar{\mathbb{Z}}_{\ell}}
\newcommand{\Qp}{\mathbb{Q}_p}
\newcommand{\Qpbar}{\overline{\mathbb Q}_p}

\newcommand{\Zpbar}{\bar{\mathbb{Z}}_{p}}
\newcommand{\Ql}{\mathbb{Q}_{\ell}}
\newcommand{\Qlbar}{\overline{\mathbb{Q}}_{\ell}}

\newcommand{\Qlpbar}{\overline{\mathbb{Q}}_{\ell'}}

\newcommand{\fisoc}[1]{\textbf{F-Isoc}(#1)}
\newcommand{\ffisoc}[1]{\textbf{F}^f\textbf{-Isoc}(#1)}
\newcommand{\fisocd}[1]{\textbf{F-Isoc}^{\dagger}(#1)}

\newcommand{\weildeligne}[2]{\textbf{WD-Rep}_{#2}(W(#1))}
\newcommand{\repcts}[2]{\textbf{Rep}_{#2}^{\text{cts}}(#1)}
\newcommand{\LLC}[1]{\textbf{LLC}(#1)}

\newcommand{\C}{\mathbb C}
\newcommand{\Z}{\mathbb Z}
\newcommand{\Q}{\mathbb Q}

\newcommand{\Xan}{X^{\text{an}}}
\newcommand{\X}{\mathfrak X}
\newcommand{\sL}{\mathcal L}

\newcommand{\piET}[1]{\pi_1^{\text{\'et}}(#1)}
\newcommand{\piTAME}[1]{\pi_1^{\text{tame}}(#1)}

\newcommand{\MdR}[2]{\mathcal{M}_{dR}((#1, \overline{#1})/#2, n, d,v)}

\newcommand{\cohrigdR}[2]{\mathcal{M}^{\text{coh. rig.}}_{dR}((#1, \overline{#1})/#2, n, d,v)}
\newcommand{\cohrig}[2]{\ensuremath{\text{CohRig}}_{d}^{n, v}(#1, #2)}

\newcommand{\rhoGEO}{\rho^{\text{geo}}}

\newcommand{\defeq}{\vcentcolon=}

\makeatletter
\newcommand{\colim@}[2]{%
  \vtop{\m@th\ialign{##\cr
    \hfil$#1\operator@font colim$\hfil\cr
    \noalign{\nointerlineskip\kern1.5\ex@}#2\cr
    \noalign{\nointerlineskip\kern-\ex@}\cr}}%
}
\newcommand{\colim}{%
  \mathop{\mathpalette\colim@{}}\nmlimits@
}
\makeatother

\usepackage{color}

\newcommand{\tocheck}[1]{{\color{black}#1}}

\newcommand\nc{\newcommand}
\nc\mf\mathfrak
\nc\mc\mathcal
\nc\mb\mathbb
\nc\msf\mathsf
\nc\mscr\mathscr

\usepackage[
backend=biber,
style=alphabetic,
sorting=nyt,
maxnames=50
]{biblatex}

\begin{document}

\title{Frobenius trace fields of cohomologically rigid local systems}
\author{Raju Krishnamoorthy}

\email{krishnamoorthy@alum.mit.edu}  

\author{Yeuk Hay Joshua Lam}
\email{joshua.lam@hu-berlin.de}
\address{Humboldt Universität Berlin,
       Institut für Mathematik- Alg.Geo.,
        Rudower Chaussee 25
        Berlin, Germany}

\date{\today}

\begin{abstract}  
Let $X/\mathbb{C}$ be a smooth variety with simple normal crossings compactification $\bar{X}$, and let $L$ be an irreducible $\Qlbar$-local system on $X$ with torsion determinant. Suppose $L$ is cohomologically rigid. The pair $(X, L)$ may be spread out to a finitely generated base, and therefore reduced modulo $p$ for almost all $p$; the Frobenius traces of this mod $p$ reduction lie in a number field $F_p$, by a theorem of Deligne. We investigate to what extent the fields $F_p$ are \emph{bounded}, meaning that they are contained in a fixed number field, independent of $p$. We prove a host of results around this question. For instance:
\begin{itemize}
\item assuming $L$ has totally degenerate unipotent monodromy around some component of $Z$, then $L$ admits  a spreading out such that the $F_p$'s are bounded;
\item without any local monodromy assumptions, we show that the $F_p$'s are bounded as soon as they are bounded at one point of $X$.
\end{itemize}

We  also speculate on the relation between the  boundedness of the  $F_p$'s, and the local system $L$ being \emph{strongly} of geometric origin, a notion due to Langer-Simpson. 
 
\end{abstract}

\maketitle 
\setcounter{tocdepth}{1}
\tableofcontents

\section{Introduction}
In this paper, we investigate an arithmetic property of rigid local systems on a complex variety $X$. More precisely, to such a local system $L$ and a prime number $p$, there is a naturally associated number field, which is by definition the field generated by the traces of  Frobenius elements at $p$. Our motivating question is whether these number fields are independent of $p$.  Our strongest result is \cref{thm:tot_degen}, which says that the $F_p$'s are indeed independent of $p$, under an assumption on the local monodromy at infinity.
\subsection{Results}
 For a smooth complex variety $Y$, we denote by $Y^{an}$ the associated complex manifold. In this article, all varieties are assumed to be geometrically connected.
 \begin{defn}
     Let $Y/k$ be a smooth variety over a perfect field. A \emph{good compactification} $\bar{Y}$ is a smooth projective variety over $k$, equipped with an open immersion $Y\hookrightarrow \bar{Y}$, such that the boundary $\bar{Y}\setminus Y\subset \bar{Y}$ is a simple normal crossings divisor. 
 \end{defn}
\begin{defn}[Esnault-Groechenig] Let $Y/k$ be a smooth variety over a perfect field, and  $j\colon Y\hookrightarrow \bar{Y}$  a good compactification with boundary divisor $Z=\cup Z_i$. Let $\ell\neq \text{char}(k)$ be a prime number. Let $L$ be a lisse $\Qlbar$-sheaf on $Y$ that is geometrically irreducible and such that the local monodromy of $L$ around each $Z_i$ is quasi-unipotent. Set $\text{End}^0(L)\subset \text{End}(L)$ to be the subsheaf of traceless endomorphisms. We say $L$ is \emph{cohomologically rigid} if
$$H^1_{\text{\'et}}(\bar{Y}_{\bar k},j_{!*}\text{End}^0(L)) = 0.$$
Essentially, this condition means that $L$ has no deformations which fix the local monodromies around the $Z_i$'s. If $k=\C$, we can make the same definition for $\C$-local systems using the Euclidean topology.

\end{defn}


Let $X/\C$ be a smooth complex variety, with associated complex manifold $X^{an}$. Fix a positive integers $n,d, v\geq 1$. We consider the finite collection
$\cohrig{\Xan}{\C}$ 
of irreducible cohomologically rigid $\C$-local systems on $\Xan$ of rank $n$, determinant of order dividing $d$, and (generalized) eigenvalues around the boundary divisors contained in the finite subgroup $\mu_{v}\subset \C^{\times}$.
Fixing an auxiliary prime $\ell$ and a (abstract, non-continuous) field isomorphism $\iota\colon \C\rightarrow \Qlbar$, this set is identified with $\cohrig{\Xan}{\Qlbar}$. 

It is a theorem of Esnault-Groechenig that for every quadruple $(\ell, n, d, v)$ as above, every element $L\in \cohrig{X^{an}}{\Qlbar}$ is integral, i.e. after picking a basepoint $x$, $L$ corresponds to a representation
$$\pi_1(\Xan,x)\rightarrow GL_n(\mc{O}_E),$$
where $\mc{O}_E$ is the ring of integers of an $\ell$-adic local field $E$. Therefore it naturally extends to a \emph{continuous} representation of the profinite completion: $\rho\colon \piET{X,x}\rightarrow GL_n(\mc{O}_E)$.  Abusing notation, we refer to this $\mc{O}_E$ local system also by $L$. 

\begin{defn}\label{def:spreading_out}
Let $X/\C$ be a smooth variety, let $\ell$ be a prime number, and let $L$ be an \'etale $\Qlbar$ local system
on $X$ whose determinant has order dividing $d$. A \emph{spreading-out} of the pair $(X,L)$ is a quintuple: $(\X,\bar{\X},\sL,S,\iota)$ where:
    \begin{itemize}
        \item $S=\Spec(R)$ is an integral affine scheme, with $\Z[1/\ell]\subset R\subset \C$, whose structure morphism $S\rightarrow \Spec(\Z[1/\ell])$ is smooth,
        \item $\bar{\X}\rightarrow S$ is a smooth projective morphism, with $\mf Z:=\bar{\X}\setminus \X$ an $S$-flat relative simple normal crossings divisor,
        \item $\sL$ is a lisse $\Qlbar$ sheaf on $\X$ whose determinant is algebraic, i.e., for all closed points $x$ of $\X$, the action of the Frobenius conjugacy class at $x$ on $\det(\mc L)$ is in $\bar{\Q}\subset \Qlbar$,
        \item $\iota$ is an isomorphism $\iota\colon \X_{\C}\rightarrow X$ such that  $\iota^*L\cong \sL|_{\X_{\C}}$.
    \end{itemize}
    We say that $\mc L$ has \emph{integral determinant up to torsion} if $\det(\mc L)$ is isomorphic to an integral power of the cyclotomic character tensored with a $d$-torsion character, i.e., there exists $m\in \Z$ such that $$(\det(\sL)\otimes \Qlbar(m))^{\otimes d}\cong \Qlbar$$ is the trivial rank 1 local system. Similarly,  we say that $\mc L$ has \emph{fractional cyclotomic determinant} if there exists integers $d, m$, with $d\geq 1$, such that $\det(\mc L)^{\otimes d}\cong \Qlbar(m)$.
\end{defn}
We will comment  on the condition of $\det(\mc{L})$ being integral up to torsion in Remark~\ref{rmk:det}: indeed, there are obvious obstructions to the boundedness of Frobenius trace fields if we remove this condition on the determinant. We now come to the second main player of our work.

\begin{defn}[Frobenius trace fields]
    Let $(\X,\bar{\X},\sL,S,\iota)$ be a spreading out of $(X,L)$, as in Definition \ref{def:spreading_out}. For any closed point $s$ of $S$, let $\X_s$ denote the fiber of $\X\rightarrow S$ over  $s$, and write $\rho_s$ for the representation of $\piET{\X_s}$ corresponding to $\mc{L}|_{\X_s}$ (after picking a basepoint).
    \begin{itemize}
    \item Let $F_s\subset \Qlbar$, the \emph{Frobenius trace field at $s$}, be the field generated (over $\mb{Q}$) by the traces of $\rho_s$ applied to  conjugacy classes of Frobenius elements at closed points of $\X_s$. 
    \item Let $\tilde{F}_s\subset F_s\subset \Qlbar$, the \emph{stable trace field at $s$}, be the intersection of the Frobenius trace fields of $\mc L|_{\X_{s'}}$, as $\Spec(\mb F_{q'})=s'\rightarrow s=\Spec(\mb F_q)$ ranges over all finite extensions of $s$ (equivalently, as $q'$ ranges over   powers of $q$).
    \item Let $F\subset \Qlbar$, the \emph{Frobenius trace field of $(\X,\sL,S,\iota)$}, be the field generated (over $\Q$) by $F_s\subset \Qlbar$, as $s$ ranges over all closed points. (Equivalently, the Frobenius trace field is the field generated by traces of $\rho$ at all closed points $s$ of $S$.) Analogously, let $\tilde{F}\subset F$, the \emph{stable Frobenius trace field of $(\X,\sL,S,\iota)$}, be the field generated by the $\tilde{F}_s$'s.
    
    \end{itemize}
    We say that $(\X,\bar{\X},\sL,S,\iota)$  has \textbf{bounded Frobenius traces} if there exists a number field $E\subset \Qlbar$ such that $F_s\subset E$  for all closed points $s$ and \emph{bounded stable Frobenius traces} if there exists a number field $\tilde{E}$ such that $\tilde{F}_s\subset \tilde{E}$ for all closed points $s$.
\end{defn}

That  $F_s$ itself  is a number field follows from  works of Lafforgue  \cite{lafforgue}  and  Deligne \cite{delignefinitude}.  We view the number fields $F_s$ as interesting and concrete arithmetic invariants of a local system; one important feature is that $F_s$ controls the \emph{companions} of the local system $\mc{L}|_{\X_s}$. (Appendix~\ref{section:companions} includes a brief summary of the theory of companions.) In this paper, we are interested in whether rigid local systems admit spreading outs with bounded Frobenius traces.

\begin{rmk}
    In the very special situation when the local system $\mc{L}$ is the cohomology of a (smooth, proper) family of varieties  over $\mf{X}$, each  field $F_s$ is simply be $\mb{Q}$. However, when we start to consider subquotients of such local systems, the fields $F_s$ can of course grow.
\end{rmk}


We first give the following simple-to-state results, Theorems \ref{thm:tot_degen}, \ref{cebotarev}, and \ref{thm:rigidifed}, to give a flavor of the statements, since they do not involve various quantifiers appearing in our main result Theorem~\ref{thm:main}. 

\begin{thmalpha}\label{thm:tot_degen}[Case of bad reduction]
    Let $X/\mb{C}$ be a smooth variety with good compactification $\bar{X}$ and boundary divisor $Z=\cup Z_i$. Let $L\in \cohrig{X}{\Qlbar}$ such that there exists $i$ and $r\geq 1$ with the following property:
    \begin{itemize}
        \item the local monodromy of $L$ around $Z_i$, quasi-unipotent by assumption, has a Jordan block with generalized eigenvalue $1$ and size $r$ that occurs with multiplicity 1.
    \end{itemize}
    Then there is a spreading-out $(\X,\bar{\X},\mc L, S, \iota)$ of $(X,L)$ with bounded Frobenius traces. Moreover, $\mc L$ is \emph{geometric in the sense of Fontaine-Mazur}, i.e., for every $\ell$-adic local field $M$ and for every map $\Spec(M)\rightarrow S$, the restriction of $\mc L$ to $\X_M$ is de Rham.
\end{thmalpha}
\begin{rmk}
      
We emphasize that Theorem \ref{thm:tot_degen} includes the case of total unipotent degeneration; in this case, $r=n$. As Petrov notes in the remarks following \cite[Theorem 2]{sasha}, the de Rham-ness does not formally follow from his main theorem. (His theorem is indeed a critical input for us.)
\end{rmk}

In the final part of the next result, we have chosen to state it for $X$ which may be defined over a number field for simplicity; the version over a finitely generated base is analogous. 

\begin{thmalpha}\label{cebotarev}
Let $X/\C$ be a smooth variety and $n,d\geq 1$. Let $L\in \cohrig{X}{\Qlbar}$.
\begin{enumerate}
\item\label{ceb:adjoint} There exists a spreading out $(\X, \bar{\X},\mc{L}, S, \iota)$ such that the associated projective local system has bounded Frobenius traces; concretely, the adjoint local system $\ad(\mc{L})\defeq \mc{L}\otimes \mc{L}^{\vee}$ has bounded Frobenius traces.
\item\label{ceb:stable} There exists a spreading out with bounded stable Frobenius traces. 
\item\label{ceb:sum} There exists an integer $h\geq 1$ such that $L^{\oplus h}$ has a spreading out with bounded Frobenius traces.
\item\label{ceb:ceb}
    Suppose   that $X$ may be defined over  $\overline{\Q}$. Then there exists a spreading out $(\X,\bar{\X},\mc L, S, \iota)$, where $S=\Spec(\mc O_K[1/N])$, with $K$ a number field and $N$ an integer, and another number field $E$, such that for a positive density\footnote{in the sense of Cebotarev density} of points $s$ of $S$, the local system $\mc L|_{\X_s}$ has Frobenius trace field contained in $E$. 
    \end{enumerate}
\end{thmalpha}
\begin{rmk}
 For  points (\ref{ceb:adjoint}), (\ref{ceb:stable}), (\ref{ceb:sum}) we can be more explicit about the number field containing all Frobenius traces: it is not difficult to see that we can take the number field where all cohomologically rigid local systems (of rank $n$ and determinant of order dividing $d$) on $X^{an}$ are defined.  Moreover, with slightly more work, one can show  that (1) and (2) hold for \emph{any} spreading out. 
\end{rmk}

\begin{rmk}
In fact, points (\ref{ceb:adjoint}), (\ref{ceb:stable}), (\ref{ceb:sum}) of Theorem~\ref{cebotarev} follow from only small parts of our  arguments for our main result Theorem~\ref{thm:main} below. Indeed, in some sense, the main point of our paper is to study the difference in Frobenius trace fields between a local system and its projectivized local system, with the former being much more delicate. 

There is a very interesting recent work of Klevdal-Patrikis \cite{klevdal2023compatibility} which shows that, for Shimura varieties in the superrigid regime and the natural local systems on them, the Frobenius trace fields of the adjoint local system is independent of $p$; in fact  they prove something much stronger. It would  be interesting to prove the boundedness   of Frobenius trace fields for the local systems themselves on such Shimura varieties.
\end{rmk}

\begin{rmk}
In response to Theorems~\ref{thm:tot_degen} and \ref{cebotarev}, the reader may ask: do general (cohomologically) rigid local systems $L$ admit a spreading out with bounded Frobenius traces? Or, as a sanity check, whether such a statement is implied by ``standard conjectures'' in this area, such as Simpson's motivicity conjecture, the Hodge conjecture, the Tate conjecture, etc. This will be discussed at length in this introduction: as we will see in Remark~\ref{rmk:all_conj}, even assuming all standard conjectures, only the statement Theorem~\ref{cebotarev} (\ref{ceb:sum}) can be deduced, and it is not at all clear whether $L$ itself should admit a spreading out with bounded traces in general. \end{rmk}

There is another situation in which we can show boundedness of Frobenius traces. After seeing our main result, Theorem~\ref{thm:main}, H\'el\`ene Esnault asked if we could show that if in a spreading-out there is a point whose Frobenius characteristic polynomials are bounded, then the whole local system has bounded Frobenius traces. Theorem~\ref{thm:rigidifed} shows that the answer is positive.

\begin{thmalpha}\label{thm:rigidifed}[Rigidifying at a point]
    Let $X/\mb{C}$ be a smooth variety. Suppose $K$ is a number field such that there is  a descent $X_K/K$ of $X$ to $K$. Suppose $\mc{L}$ is a local system on $X_K$ with fractional cyclotomic determinant whose restriction to $X$ is an element $L\in \cohrig{X}{\Qlbar}$. Suppose that there exists a rational point $x\in X_K(K)$ and a number field $E\subset \Qlbar$ such that $\mc{L}|_x$ has all coefficients of characteristic polynomials of Frobenius conjugacy classes in $E$. Then after possibly passing to a finite extension $K'/K$, $\mc{L}$ itself has bounded Frobenius traces.
\end{thmalpha}
The proof of this is given in \ref{section:rigidify}, where as usual there is a more general version where $X$ is descended to a finitely generated field, instead of  a number field.

We now state our main result. 

\begin{thmalpha}\label{thm:main}
Suppose  $L\in \cohrig{X}{\Qlbar}$. Then there exists a spreading out $(\X,\bar{\X}, \mc{L}, S, \iota)$ with integral determinant up to torsion. For any such spreading out, there exists integers $N$ (depending on $X,n, d$), $\ell_0$ (depending on $\X, n, d$), and a number field $E$ (depending on $\X, n, d$) such that the following hold:

    \begin{enumerate}
        \item 
For any closed point $s$ of $S$, $[F_s: \mb{Q}]\leq N$; 
        \item 
For any closed point $s$ of $S$, and for any prime  $m>\ell_0$ different from the characteristic of $s$, $F_s$ is unramified over $m$; 

\item\label{mainthm:3} replacing $S$ by a scheme \'etale over it if necessary, there exists a number field $E\subset \Qlbar$, and $\alpha=\alpha(X, n, d) \in \mb{Z}_{\geq 1}$, such that for any prime number $p$ and any $s\colon \Spec(\mb F_q)\rightarrow S$ with $\mb F_q$ containing $\mb F_{p^{\alpha}}$, the trace field of $\mc L|_{\X_s}$ is contained in $E$. 



    \end{enumerate}
\end{thmalpha}
As may be seen, point (2) of the theorem is about the ramification of $F_s$ away from $\text{char}(s)$, the characteristic of $s$,  and the proof of  (3) will proceed by analyzing the ramification of $F_s$ \emph{at} $\text{char}(s)$; for this reason we sometimes refer to (2) as the theorem away from $\text{char}(s)$ and to (3) as the theorem at $\text{char}(s)$. What we find most surprising about Theorem \ref{thm:main} is the uniformity in $\alpha$; indeed, $\alpha$ depends only on $X, n, d$.

We now make several remarks on Theorem~\ref{thm:main}.
\begin{rmk}[Cautionary example]\label{rmk:sq_root}
    One important remark is that point (3) of  Theorem \ref{thm:main} is certainly not true for arbitrary $\mb{F}_q$ points of $S$,  i.e. it is possible that $F_s$ is ramified at $p = \text{char}(\kappa(s))$ for infinitely many $s$. Let $K$ be any number field such that $G_K$ has a character $\chi$ which is  a square root of the $\ell$-cyclotomic character. Consider the Legendre family of elliptic curves on $X_K\defeq \mb{P}_K^1\setminus \{0, 1, \infty\}$, with associated $\ell$-adic local system $L_K'$ on $X_K$; now define $L_K\defeq L_K'\otimes \chi$ admits a spreading out $\mc{L}$ to some $\mf{X}$, and for any $s: \Spec(\mb{F}_p)\rightarrow \mf{X}$, the Frobenius trace field of $\mc{L}|_s$ is $\mb{Q}(\sqrt{p})$.
    
    Indeed, the unramifiedness of $F_s$ at $p$ for arbitrary $\Spec(\mb{F}_q)$ points $s$ would imply that the Frobenius traces are bounded\footnote{as there are only finitely many number fields with bounded degree and ramification}.
\end{rmk}

\begin{rmk}\label{rmk:det}
    In Definition~\ref{def:spreading_out}, we have specified that $\det(\mc{L})$ differs from an integral power of the cyclotomic character by a torsion character. In some sense this is a somewhat arbitrary condition: in principle one should be able to relax this condition, for example to requiring that $\det(\mc{L})$ has a power which is \emph{de Rham} at all places above $\ell$, or simply that $\det(\mc{L})$ has bounded Frobenius traces. However, these conditions seem equally unnatural to us: a priori, in light of  Remark~\ref{rmk:sq_root}, one might hope that putting a condition on $\det(\mc{L})$ (e.g. a condition on the Hodge-Tate weghts) might pick out a twist of $\mc{L}$ with bounded Frobenius traces; however, it is not even clear to us whether such a twist should exist (see Section~\ref{section:discuss} for more on this), even assuming all  conjectures.  
\end{rmk}
We now give the proof of Theorem~\ref{cebotarev}, which follows straightforwardly from Theorem~\ref{thm:main}.
\begin{proof}[Proof of Theorem~\ref{cebotarev} given Theorem~\ref{thm:main}]
    The proofs of Point \ref{ceb:adjoint} and \ref{ceb:sum} will be  given in Section~\ref{section:adjoint}. Point (2) is immediate from point (3) of Theorem~\ref{thm:main}.

    We now give the proof of point (\ref{ceb:ceb}) of Theorem~\ref{cebotarev}. 
    Indeed, this  follows from Theorem \ref{thm:main}(\ref{mainthm:3}) and the following corollary of Cebotarev: for any number field $K$, finitely generated subring $\mc O\subset K$, and positive integer $\alpha$, there exists a finite extension $K'/K$ and a finitely generated subring $\mc O'\subset K'$ that contains $\mc O$, such that for a positive proportion of the prime ideals $\mf p'$ of $\mc O'$, the residue field $\mc O'/\mf p'$ has degree over $\Z/(\mf p'\cap \Z)$, the underlying prime field, divisible by $\alpha$. 
\end{proof}

\subsection{Discussions}\label{section:discuss}
The motivicity conjecture of  Simpson \cite[p.9]{simpsonhiggs} states that each $L\in \cohrig{\Xan}{\C}$  is \emph{of geometric origin} (also referred to as \emph{motivic} in the following): this roughly means that after restricting to a dense open subset $U\subset X$, it arises as a subquotient (or even summand) in the cohomology of a family of smooth projective varieties. Note that if $L$ is of geometric origin, with family $\pi\colon \mc Y\rightarrow U$ and $i\geq 0$ such that $L|_U$ is a summand of $R^i\pi_*\C$, then it is very well possible that $L|_U$ is \emph{not cut out by an algebraic cycle}.

Therefore, there is an a priori more restrictive notion, namely that $L$ is cut out by algebraic cycles from the cohomology of a family of varieties; this notion was introduced by Langer-Simpson \cite{langersimpson} (see the end of Section 2.5 of loc.cit.), and we follow their terminology and call such local systems  \textbf{strongly of geometric origin}, or \textbf{strongly motivic}. It is straightforward to see that if $L$ is strongly motivic, then a spreading out with bounded Frobenius traces does exist.

\begin{question}\label{question}
Suppose $X$ is a smooth complex variety, and  $L$ a local system on $X$. Does the existence of a spreading out $(\X, \bar{\X}, S, \mc{L}, \iota)$ with bounded Frobenius traces imply that $L$ is strongly of geometric origin?
\end{question}

\begin{rmk}
 Unfortunately,  it seems hard to probe Question~\ref{question} since there are no known examples of local systems of geometric origin which are not strongly so. We note that local systems which are strongly motivic admit spreading outs satisfying other properties, such as being \emph{de Rham} at all places above $\ell$, which a priori seems independent from having bounded Frobenius traces. Indeed, the relative Fontaine-Mazur conjecture \cite{liuzhu} together with the Tate conjecture imply that a local system is strongly motivic if it admits a spreading out which is de Rham at $\ell$. 

Note also that, conjecturally, the subtlety between geometric and strongly geometric local systems disappear once we consider the projectivized local system. Indeed, in his beautiful work, Petrov \cite{sasha} shows that the adjoint of any geometrically irreducible arithmetic $\Qlbar$-local system is de Rham at places above $\ell$, and hence this adjoint local system satisfies the assumptions of the relative Fontaine-Mazur conjecture, and therefore is conjecturally strongly motivic. 
\end{rmk}

\begin{rmk}[Assuming all conjectures]\label{rmk:all_conj}
Langer-Simpson say the following \cite[p.1544]{langersimpson}: ``(If we) assume the Hodge conjecture, then one might be able to show that $V$ of weak geometric origin implies that some $V^{\oplus k}$ is of strong geometric origin.'' 

The Hodge conjecture implies that the Hodge realization functor $\textbf{Mot}_{\Q}\rightarrow \textbf{HS}_{\Q}$ from motives to Hodge structures is fully faithful. More generally, it is a folklore conjecture that if $S/\C$ is a smooth variety, then the functor:
$$\textbf{Mot}(S)_{\Q}\rightarrow \Q\textbf{-VHS}(S)$$
from $\Q$-motives over $S$ (appropriately defined) to $\Q$-VHS is fully faithful. Therefore, for any finite extension $E/\Q$, the base-changed functor $\textbf{Mot}(S)_{E}\rightarrow E\textbf{-VHS}(S)$ is conjecturally also fully-faithful.

Suppose now that $L$ is a $\C$-local system on $S$ of geometric origin. Then there exists 
\begin{itemize}
    \item a number field $E$ such that $L$ may be conjugated to have coefficients in $E$, and
    \item a open dense $U\subset S$, a smooth projective family $\pi\colon T\rightarrow U$, and $i\geq 0$
\end{itemize} such that $L|_U$ is a summand of $R^i\pi_*E$. We may consider the \emph{isotypic component} $(L|_U)^{\oplus k}$ of $L|_U$ in $R^i\pi_*E$; this will be a sub $E$-VHS. Hence assuming above folklore conjecture, $(L|_U)^{\oplus k}$ is cut out by an algebraic cycle, i.e., $L^{\oplus k}$ is strongly of geometric origin.

Combining this with Simpson's motivicity conjecture,  we can therefore expect that, for $L$ cohomologically rigid,  there exists a $k\geq 1$ such that $L^{\oplus k}$ has a spreading out with bounded Frobenius traces; this expectation is confirmed by Theorem~\ref{cebotarev}(\ref{ceb:sum}).     
\end{rmk}

For the case of rank two local systems, many cases of Question~\ref{question} have been answered in the positive by \cite{snowdentsimerman} and \cite{rajuconstruct}. However, this case is not representative since motivic rank two local systems arise in the  cohomology of a family of abelian varieties, and in this case motivic implies strongly motivic, by the work of Moonen \cite[Theorem 6.4]{moonenlinear}. To illustrate the subtlety, note that if $X/\C$ is smooth projective, $\pi\colon A\rightarrow X$ is an abelian scheme, and $L$ is an irreducible summand of $R^1\pi_*\C$, then the claim is that there exists \emph{another abelian scheme} $B\rightarrow X$, together with endomorphisms $R\hookrightarrow \text{End}^0(B)$, such that $L$ is a component under the decomposition of the first cohomology under $R$.\footnote{The arithmetic analog of this discussion is the following. An \emph{arithmetic local system} is an \'etale $\Qlbar$-local system $L$ on $X=X_{\C}$ that is a sub-quotient of a local system $M$ that extends to a local system $\mb M$ on $X_K$, where $K$ is a finitely generated field. (For simplicity, we will assume that $L$ is a sub-representation of $M$.) It turns out that, after extending $K$ by a finitely  extension, $L$ also extends to an \'etale local system $\mb L$  on $X_K$. However, there is \emph{no reason} that  there is an arithmetization $\mb L$ for which  the map $L\rightarrow M$ is $G_K$-equivariant.} 

\subsection{Geometric applications}
As explained in \cite[Remark 1.9]{rajuconstruct}, our main theorem combined with the main theorem of \emph{loc. cit.} provides an arithmetic approach to the following special case of the Corlette-Simpson theorem \cite{corlettesimpson}: cohomologically rigid rank 2 local systems on quasi-projective varieties with infinite monodromy at $\infty$ and trivial determinant come from families of abelian varieties.
\subsection{Sketch of proof of Theorem~\ref{thm:main}}
For this proof sketch we assume $d=1$, i.e. that our local systems, when restricted to the geometric fundamental group, have trivial determinant. The overall strategy is to analyze the ramification behavior of the trace field $F_s$ at each prime $m$ by looking at the $m$-adic companions.

For point (1) note that the degree of the trace field is the number of $\ell$-companions;   that $\ell$-companions of cohomologically rigid local systems are again cohomolgically rigid was proven by Esnault-Groechenig, and we may conclude since there are only finitely many such local systems with fixed rank and trivial determinant.

Let us write $p=\text{char}(s)$ and $k(s)$ for the residue field at $s$. For the proof of  point (2), we let $\ell_0=\ell_0(X, n,d)$ be a prime such that for all prime numbers $m>\ell_0$, the topological, cohomologically rigid, local systems on $X^{an}$ all have coefficients unramified at $m$. For a spreading out $(\mf{X}, \mc{L}, S, \iota)$, a point $s\in S$, and $m\neq p$, let $\mc{L}_m$ be any  $m$-companion of $\mc{L}|_{\mf{X}_s}$, which is an $m$-adic local system. Then $\mc{L}_m|_{\X_{\bar{s}}}$ also has unramified coefficients. We then show that, after possibly discarding finitely many $m$'s,  we may descend  this local system to $\X_{s}$ itself, with  \emph{unramified coefficients}: indeed, it is immediate that there is a descent as a Weil sheaf on $\X_s$ (i.e. a representation of the Weil group), but to get an \'etale sheaf we need an argument involving absolute irreducibility of residual representations. This descent will not, in general, be the local system we started with: rather, it is a \emph{twist} of the original by a character of the base $\Spec(k(s))$; finally, since they are both \'etale sheaves, one sees that the twisting cannot add ramification to the trace field at $\ell$, as desired. The details are given at the end of Section~\ref{section:unramified}. 

For point (3) of Theorem~\ref{thm:main}, as mentioned above, we proceed by showing that $F_s$ is also unramified at $m=\text{char}(s)=p$, at least under some assumptions on the point $s$, which we now explain. We again analyze the $p$-companions of $\mc{L}|_{\X_s}$, which in this case means $F$-isocrystals on $\X_s$: these were shown to exist by Esnault-Groechenig \cite{eg}.  Let $\mc{F}$ be such a $p$-companion. In the ideal world one would want to show that $\mc{F}$ also has unramified coefficients: this cannot be true in this generality, as Remark~\ref{rmk:sq_root} shows, and we will instead show that it is true as long as $s$ is an $\mb{F}_q$ point with $q$ divisible by some fixed integer $\alpha=\alpha(X, n,d)$, i.e. after extending $s$. After extending $s: \Spec(\mb{F}_q)\rightarrow S$ to some $s': \Spec(\mb{F}_{q'})\rightarrow S$, the work of Esnault-Groechenig (interleaved with the work of de Jong-Esnault) provides us with a collection of cohomologically rigid $F$-isocrystals on $\X_s$ which have unramified coefficients by construction. The restriction $\mc{F}|_{\X_{s'}}$ is isomorphic to one of these $F$-isocrystals up to twisting by a $F$-isocrystal from the base $\Spec(\mb{F}_{q'})$. Assuming $q'$ is sufficiently divisible, one can show that this twisting does not affect the unramifiedness of the coefficients, which lets us conclude. 

\subsection{Related works}
As mentioned above, the work most closely related to ours is perhaps that of Klevdal-Patrikis \cite{klevdal2023compatibility}. The works of Snowden-Tsimerman \cite{snowdentsimerman} and Krishnamoorthy-Yang-Zuo \cite{rajuconstruct}  prove motivicity of rank two local systems given boundedness of Frobenius traces, under a mild condition on local monodromies. 
In \cite{koji}, Shimizu proved that Frobenius traces are bounded for local systems satisfying the assumptions of the  Fontaine-Mazur conjecture, assuming the latter as well as the Generalized Riemann Hypothesis.  In light of Petrov's conjecture \cite[Conjecture 1 bis]{sasha} that arithmetic local systems are of geometric origin, it would be interesting to investigate boundedness of Frobenius trace fields of arithmetic local systems, though this seems somewhat out of reach. 


\subsection{Outline}
In Section~\ref{section:spread} we collect together  results of various authors about spreading out rigid local systems. Section~\ref{section:degree} proves point (1) of Theorem~\ref{thm:main} concerning the degree of Frobenius trace fields. Sections~\ref{section:unramified} and \ref{section:fcrystals}, which make up the core of the proof of Theorem~\ref{thm:main},  handle points (2) and (3) of Theorem~\ref{thm:main} respectively: the former concerns ramification away from $p$ and the latter at $p$.

Finally, in Section~\ref{section:adjointquestion} we deduce Theorems~\ref{cebotarev}, \ref{thm:rigidifed}, and \ref{thm:tot_degen}; in particular, we derive results about the stable trace field, the adjoint local system, $L^{\oplus h}$, the rigidified problem, and stronger results in the case of bad reduction from our main theorem. 
In Appendix \ref{section:companions}, we quickly review the theory of companions and prove several lemmas that we need for our main results.


\subsection{Notations and conventions}
For any field $k$, we denote by $G_k$ its absolute Galois group, and for a scheme $Z$ we write $\piET{Z}$ for its \'etale fundamental group. For maps of schemes $Z\rightarrow S$ and $S'\rightarrow S$, we denote by $Z_{S'}$ the basechange $Z\times_S S'$. The symbol $q$ will always denote a power of a prime number $p$, and $\mb{F}_q$ denotes the finite field with $q$ elements. The notation $N=N(A, B, C, \cdots) \in \mb{Z}$ means that the integer $N$ depends only on $A, B, C, \cdots$.  For a point $s: \Spec(\mb{F}_q)\rightarrow S$, $\bar{s}$ denotes a geometric point $\bar{s}: \Spec(\overline{\mb{F}}_q)\rightarrow S$ over $s$.

\subsection*{Acknowledgments} Most of this work builds on the ground-breaking works of Esnault-Groechenig \cite{egintegral, eg, eg2}. We are grateful to H\'el\`ene Esnault, Bruno Klingler, Daniel Litt, and Sasha Petrov for many  enlightening discussions. Lam was supported by a Dirichlet Fellowship during the course of this work. Krishnamoorthy was supported by the European Research Council (ERC)
under the European Union’s Horizon 2020 research and innovation program, grant agreement
no. 101020009, project TameHodge.

\section{Spreading out}\label{section:spread}

As in the introduction, $\cohrig{\Xan}{\Qlbar}$ denotes the collection of rank $n$ cohomologically rigid local systems with quasi-unipotent monodromy around the boundary divisors (with generalized eigenvalues in $\mu_v$) whose determinant is torsion of order dividing $d$. For any variety $X$ over a perfect field $k$ and prime $\ell$ different from the characteristic of $k$, we define   the slight variant $\cohrig{X}{\Qlbar}$, which is the collection of rank $n$, cohomologically rigid,  \'etale $\Qlbar$-local systems, with the same boundary monodromy condition, whose determinant is torsion of order dividing $d$; for a $\ell$-adic local field $E$, we denote by $\cohrig{X}{E}$ the collection of $E$-local systems with the same adjectives as above. 
\subsection{}
The goal of this section is to prove the following
\begin{lem}\label{lem:spreading_out}
    Let $X/\C$ be a smooth variety, with good compactification $\bar{X}$  and fix integers $n,d, v\geq 1$. Then there exists a spreading out $(\X, \bar{\X})\rightarrow S$ of the pair $(X,\bar{X})$ such that for every  $L\in \cohrig{X}{\Qlbar}$, there is an extension of $L$ to an \'etale local system $\mathcal L$ on $\X$ with the further property that $\det(\mathcal L)$ is torsion of order $d$. 
\end{lem}

This result is certainly known to experts, and we merely collect together the arguments of various authors, as we will indicate.
As  $\cohrig{X}{\Qlbar}$ is a finite set, it suffices to show that for each $L$, there exists such a spreading out: this is done in Lemma~\ref{lemma:spreadfinitetypebase}. In Section~\ref{section:descentfgfield} we first show that we can spread out $X$ and $L$ to a finitely generated field, and then in Section~\ref{section:spreadfinitetypebase} we show that this can be extended to a finite type base. 


\subsection{Descent}\label{section:descentfgfield}
\begin{prop}\label{prop:rep_order_d}
    Let $X/K$ be a smooth variety over a field of characteristic 0, let $E_{\ell}$ be an $\ell$-adic local field, and let $\mathcal O$ be the ring of integers in $E_{\ell}$. Let $\rho\colon \piET{X_K}\rightarrow \GL_n(\mc O)$ be an \emph{abstract} representation (i.e. not assumed to be continuous), where $n$ is coprime to $\ell$, such that the restriction to $\piET{X_{\bar K}}$ has determinant of order $d\geq 1$. Then there exists a finite extension $K'/K$ and an abstract representation $\rho'\colon \piET{X_{K'}}\rightarrow \GL_n(O)$ such that $$
    \rho|_{\pi_1(X_{\bar K})}\cong \rho'|_{\pi_1(X_{\bar K})}$$
    and $\rho'$ has determinant of order $d$.
\end{prop}
\begin{proof} 
    First of all, by extending $K$, we may assume that $X$ has a $K$-point $x$.

    Consider the representation $\rho|_x\colon G_K\rightarrow \GL_n(\mc O)$.  We claim that after another finite extension of $K$, the character $\det(\rho|_x)\colon G_K\rightarrow \mc O^{\times}$ has an $n^{\text{th}}$ root. Indeed, replacing $K$ by a finite extension if necessary, we may assume that $\det(\rho_x)$ is trivial mod $\pi$, where $\pi$ is a uniformizer of $\mc O$. Then, as $(n,\ell)=1$, it follows from Hensel's lemma/the binomial theorem that there is a well-defined (continuous) group homomorphism:
    $$\sqrt[n]{-}\colon 1+\pi\mc O\rightarrow \mc O^{\times}$$
    The tensor product $\rho \otimes \big(\sqrt[n]{\det(\rho|_x)}\big)^{-1}$ will therefore have determinant order $d$, as required.
\end{proof}

\begin{prop}\label{prop:simpson}
    Let $X/\mathbb C$ be a smooth variety with good compactification $\bar{X}$. Let $L$ be an element of $\cohrig{X}{\Qlbar}$. Then there exists a finitely generated subfield $K\subset \mb{C}$ and a pair $(\mc{X},\bar{\mc{X}})/K$ such that $(\mc{X},\bar{\mc{X}})\otimes \mb{C}\simeq (X,\bar{X})$, and  $L$ extends to a local system $\mb L$ on $\mc{X}/K$ with determinant of order dividing $d$. 
    \end{prop}

As mentioned before, this is a slight enhancement of an argument of Simpson's, and very similar arguments have appeared in \cite{egintegral, litt}: see for example \cite[Proposition 3.1.1]{litt} and the references therein. 
    
\begin{proof}
For ease of notation we will give the proof in the case $d=1$. The case of general $d$ is essentially identical. 

Suppose that $L$ has coefficients in the ring of integers $\mc{O}\subset E_{\ell}$; we view $L$ equivalently as a free $\mc{O}$-module $M$ of rank  $n$ with an action of $\piET{X}$ via $\piET X\rightarrow \GL(M)$.

Let $K\subset \mb{C}$ be a finitely generated field, and $(\mc X,\bar{\mc X})/K$ such that $(\mc X, \bar{\mc X})\otimes \mb{C}\simeq (X,\bar{X})$.  

For each $g\in G_K$, let $\ ^{g}M$ denote the representation $\rho \circ Ad(g)$. As in \cite[Proof of Theorem 4]{simpsonhiggs}, we assume that there exists $c(g)\in \GL_n(E_{\ell})$  such that $M\otimes \mb{Q}$ and $\ ^gM\otimes \mb{Q}$ are conjugate by $c(g)$; the element $c(g)$ is well defined up to a scalar in $E_{\ell}^{\times}$.

As $\rho$ is irreducible, it follows from a Jordan-Zassenhaus type theorem that there are only finitely many homothety classes of $\piET X$-invariant lattices inside $M\otimes \mb{Q}$ \cite[Theorem 1.1]{suh}.
We claim that $G_K$  acts on  this finite set of homothety classes. Indeed, let $N\subset M\otimes \mb{Q}$ be a $\piET X$-stable lattice. Then we have an integral representation $\piET X\rightarrow \GL(N)$. Conjugating by each $g\in G_K$, we obtain a (potentially non-isomorphic) integral representation; however, by our assumption, the rational isomorphism class of this new representation is simply $M\otimes \mb{Q}$.  Therefore $G_K$ acts on the set of \emph{isomorphism classes} of integral $\piET X$-subrepresentations of $M\otimes \mb{Q}$, i.e. the set of homothety classes of stable lattices. By taking a finite extension of $K$ we now assume the action is trivial, and by scaling each $c(g)$ appropriately we may therefore  assume 
\[
c(g): M\xrightarrow{\simeq} \ ^gM.
\]
as $\piET{X}$-representations.

We claim that from the $c(g)$, we may construct a class $\xi \in H^2(G_K,\mc O^{\times})$, given by the following formula:
$$\xi(g,h)=c(gh)^{-1}\circ \ ^gc(h)\circ c(g)\colon G_K\times G_K\rightarrow \text{Aut}_{\piET{X_{\bar K}}}(M)\cong \mc O^{\times}.$$
(Again, here $\mc O^{\times}$ is equipped with the \emph{trivial} $G_K$ action.)

For the moment, let us assume that $\xi = 0$ and see how to conclude. Indeed, $\xi =0$ implies that the representation extends to a \emph{not necessarily continuous} representation $\piET{X_K} \rightarrow GL(M)$. (This is a general feature of semi-direct products.) Note that the restriction to the geometric fundamental group is continuous, and moreover, the projectivization is continuous by Simpson's argument.

Using Proposition \ref{prop:rep_order_d}, by increasing $K$, we may modify our representation such that its determinant has order $d$. Call this representation $\rho$. Let $\mu_d\GL(M)$ denote the subgroup of elements of $\GL(M)$ with  $d$-torsion determinant, so that $\rho$ has image in $\mu_d\GL(M)$. Then the natural map of $\ell$-adic groups:
$$\pi\colon \mu_d\GL(M)\rightarrow PGL(M)$$
is finite-to-one, and moreover is a local homeomorphism by \cite[Lemma 5.3]{conradadelic} (as the valuation ring in any algebraic extension of $\Ql$ is Henselian.). As we know that the composition $\pi\circ \rho$ is continuous, it follows that $\rho$ is continuous. Therefore, the $\rho$ we have constructed satisfies the conclusion of the Proposition.

Furthermore, note that the injection of trivial $G_K$-modules: $\mc O\hookrightarrow \Zlbar$ splits, and hence $\xi$ vanishes if and only if the image class in $H^2(G_K,\Zlbar^{\times})$ vanishes. Abusing notation, we call this image class $\xi$. 

Finally, we show that we can kill the Brauer class $\xi$ by passing to a finite extension. Let us denote by $PL$ the projective local system associated to $L$. 
    
    By  \cite[Proof of Theorem 4]{simpsonhiggs}, $PL$ extends to a local system on $\mc{X}/K$; let us denote this by $P\mb{L}$.  (Note that \cite[Theorem 4]{simpsonhiggs} does not assume that $X$ is proper.) Moreover, by enlarging $K$ if necessary, we may assume that the torsion character $\det(L)$ extends to a rank one local system on $\mc{X}/K$. 

    On the other hand, the obstruction to $\det(L)$ extending is $c^n$, where the latter denotes the image of $c$ under the map
    \[
    H^2(G_K, \Zlbar^{\times}) \rightarrow H^2(G_K, \Zlbar^{\times})
    \]
induced by the short exact sequence  
\[
0\rightarrow \mu_n \rightarrow  \Zlbar^{\times} \xrightarrow[]{x\mapsto x^n} \Zlbar^{\times} \rightarrow 0.
\] 
By the previous discussion we have  $c^n=0$, and hence $c$ comes from a class $b$ in $H^2(G_K, \mu_n)$. Now, by \cite[Proposition 2.2.16]{sharifi}, $b$ lies in the image  of the inflation map $\mathrm{Infl}:H^2(G_K/N, \mu_n)\rightarrow H^2(G_K, \mu_n)$, for some finite index normal subgroup  $N$; finally, the composition $H^2(G_K/N, \mu_n)\xrightarrow[]{\mathrm{Infl}} H^2(G_K, \mu_n) \xrightarrow[]{\mathrm{Res}} H^2(N, \mu_n)$ is zero, and so we may replace $K$ by the finite extension defined by $N$ to kill the class $c$, as required.
\end{proof}
\subsection{Spreading out over a finite type base}\label{section:spreadfinitetypebase}
The following is a slight generalization of \cite[Proposition~6.1]{sasha}, which is attributed to  Litt, as well as to Liu-Zhu \cite{liuzhu}. 
\begin{lem}\label{lemma:spreadfinitetypebase}
    Let $X/\mathbb C$ be a smooth variety with good compactification $\bar{X}$. Fix integers $n, d\geq 1$ and $L \in \cohrig{X}{\Qlbar}$. Then there exists a smooth $\Z$-algebra $A\subset \C$ and a spreading-out $(\X, \bar{\X})\rightarrow \Spec(A)$ such that $L$ extends to a local system $\mathcal L$ on all of $\X$; moreover, we may choose $\X\rightarrow \Spec(A)$ such that $\mathcal L$ has torsion determinant of order $d$.
\end{lem}
\begin{proof}
    We denote by $\rhoGEO: \piET{X}\rightarrow \GL_n(\mc{O})$ the representation corresponding to $L$, where $\mc{O}$ is the ring of integers of an $\ell$-adic local field $E_{\ell}$; here we have used the integrality of rigid local systems.

         By Proposition~\ref{prop:simpson}, there exists a finitely generated field $K\subset \C$ such that $(X,\bar{X})$ descends to $K$ and moreover $\rhoGEO$ descends to a representation $\rho$ of  $\piET{X_{K}}$, whose determinant has order $d$. By enlarging $E_{\ell}$ if necessary, we assume that $\rho$ still takes values in $\GL_n(\mc{O})$. There exists an integrally closed, finitely generated, subring $A'\subset K$ such that $X$ extends to a smooth proper scheme $\X/\Spec(A')$, and $\det(\rho)$ extends to $\X$.

            By enlarging $K$, we may assume we have  a point $x\in X(K)$, and moreover that $\bar{\rho}:\piET{X_K}\rightarrow \GL_n(\mc{O}/\ell^2)$ is trivial. Let $\rho_x: G_K\rightarrow \GL_n(\mc{O})$ denote the corresponding representation.
            
             Let $\mathcal T$ be the union of the set of prime numbers dividing $|\GL_n(\mc{O}/\ell^2)|$ and the set of  height one primes where $\det{\rho}$ is ramified. (Note that this means $\mathcal T$ contains the set of primes that divide $|\GL_n(\mc{O}_{E_{\ell}}/\ell^N)|$ for all $N\geq 1$.)
            \item   Let $A\defeq A'[\frac{1}{\mc{T}}]$,  and denote by $\X_A$ the basechange of $\X$ to $\Spec(A)$. Let $\mf{p}$ be a height one prime ideal of $A$, with associated discrete valuation ring $A_{\mf{p}}$, whose fraction field we denote by $K_{\mf{p}}$; let $p$ denote the characteristic of the  residue field of $A_{\mf{p}}$. We first show that $\rho_x$ is unramified at $\mf{p}$. 
            
            The representation
            $\rho\colon \piET{X_K, x}\rightarrow \GL_n(\mc{O})$
            factors through the prime-to-$p$ quotient $\piET{X_K, x}^{(p')}$.
            By local constancy of the prime-to-$p$  fundamental group, $G_{K_{\mf{p}}}$ acts on $\piET{X_K, x}^{(p')}$ through the unramified quotient. By the same argument as in \cite[Proof of Proposition 6.1]{sasha}, we deduce that $\rho_x$ is unramified at $v$. 
            
            We now show that $\rho$ extends to a local system on $\X_A$. For each $j\geq 1$, let $K_j/K$ be the extension trivializing $\rho_x \bmod 
            \ell^j$, and let $A_j$ be the integral closure of $A$ in $K_j$; the above discussion shows that $A_j$ is \'etale over $A$. As in loc.cit., to show that $\rho$ extends to $\mf{X}_A$, it suffices to show that it extends to $\mf{X}_{A_j}$: indeed, this follows from the fact that $A\rightarrow A_j$ is an \'etale ring map. That $\rho$ extends to $\X_{A_j}$ follows directly from \cite[Tag 0EYJ]{stacksproject}. 
\end{proof}

\section{Degree of trace fields}\label{section:degree}

\subsection{} Colloquially, the following lemma says the following: every cohomologically rigid $\Qlbar$-local system on $X$ with fixed $(n,d,v)$ may be defined over a finite extension $E_{\ell}/\Ql$, and moreover, we may relate these local systems to local systems in characteristic $p$ for any spreading out of $X$. This lemma is an easy corollary of the technique of Esnault-Groechenig \cite{egintegral}. 
\begin{lem}\label{lemma:biject}
    Let $X/\C$ be a smooth variety with good compactification $\bar{X}$ and let $n,d\geq 1$ be integers. 
    \begin{enumerate} \item For any prime number $\ell$,  there exists an $\ell$-adic local field $E_{\ell}$, equipped with an embedding $E_{\ell}\xhookrightarrow{} \Qlbar$, and an integer $N_1=N_1(X,n,d)$ such that for any spreading out $(\X,\bar{\X})/S$ of $(X,\bar X)$ and for all primes $p>N_1$ and all closed points $s$ of residue characteristic $p$, there is the following natural diagram: 
    $$\xymatrix{
  \cohrig{\X_{\bar s}}{\Qlbar}\ar[r]^{\simeq}\ar@{=}[d] & \cohrig{X}{\Qlbar}\ar@{=}[d]\ar@{=}[r]^{\text{EG}} & \cohrig{\Xan}{\Qlbar}\ar@{=}[d]\\
\cohrig{\X_{\bar s}}{E_{\ell}}\ar[r]^{\simeq} &\cohrig{X}{E_{\ell}}\ar@{=}[r] & \cohrig{\Xan}{E_{\ell}},  
    }$$
    where all of the above maps are isomorphisms.

  \item   Furthermore, there exists $N_2=N_2(X, d, n)$, such that for $\ell> N_2$, we may take $E_{\ell}$ to be unramified over $\mb{Q}_{\ell}$.   
    \end{enumerate}
\end{lem}
\begin{proof}
\begin{enumerate}
\item 
    The set $\cohrig{X^{an}}{\Qlbar}$ is finite, and we may take a local field $E_{\ell}$, equipped with $E_{\ell}\xhookrightarrow{} \Qlbar$ such that all the elements of $\cohrig{X^{an}}{\Qlbar}$ have coefficients in $E_{\ell}$, i.e. the natural map $\cohrig{X^{an}}{E_{\ell}}\rightarrow \cohrig{X^{an}}{\Qlbar}$ is bijective. By \cite[Theorem 1.1]{egintegral}, we may assume that each element of $\cohrig{X^{an}}{E_{\ell}}$ is in fact integral, so that the natural map $\cohrig{X}{E_{\ell}}\rightarrow \cohrig{X^{an}}{E_{\ell}}$ is also bijective. 

     Let $N_1$ be such that, for all primes $p>N_1$, every finite quotient of $\GL_n(\mc{O}_{E_{\ell}})$ has order prime to $p$. 
    
    Given a spreading-out $\mf X\rightarrow S$ of $X$, a point $\eta\in  S$, and a specialization of points $\eta\leadsto s$, for a point $s\in  S$ of residue characteristic $p$, there is a surjective specialization map 
    \begin{equation}\label{specialize}
    \text{sp}^{(p')}\colon \piET{\X_{\eta}}\twoheadrightarrow \piET{\X_{s}}^{(p')},
    \end{equation}
    where the superscript $(p')$ refers to the prime-to-$p$ quotient of $\piET{\X_s}$,
    giving rise to a map $\piET{X}\cong \piET{\X_{\bar \eta}}\twoheadrightarrow \piET{\X_{\bar s}}^{(p')}$ which in fact induces an isomorphism $\piET{X}^{(p')}\cong\piET{\X_{\bar s}}^{(p')}$.
    
    Therefore, for $p> N_1$, there is a natural bijection between rank $n$ local systems with coefficients in  $\mathcal{O}_{E_{\ell}}$ on $\X_{\bar s}$ (which automatically have prime-to-$p$ image) and those on $X$. Then, exactly as in \cite[Proof of Theorem 1.1]{egintegral}, it follows from local acyclicity, \cite[Lemma 3.14]{saito17}, that this bijection preserves the set of cohomologically rigid local systems. 
    
    
    \item We may simply take $N_2>N_1$ such that the field of definition of  every element of $\cohrig{X^{an}}{\Qlbar}$ is unramified at all $\ell>N_2$. 
    \end{enumerate}\end{proof}

    


\begin{lem}\label{lemma:companionclosed}
    Let $Y/k$ be a smooth variety over a finite field. Then the collection of \tocheck{tame} $\Qlbar$-local systems on $Y$ with the following properties:
    \begin{itemize}
        \item rank $n$,
        \item determinant of order $d$,
        \item geometrically irreducible, and 
        \item cohomologically rigid
    \end{itemize}
    is finite and closed under $\tau$-companions for any abstract field isomorphism $\tau\colon \Qlbar\rightarrow \Qlbar$. 
\end{lem}
\begin{proof}
The finiteness of this set follows from \cite[Theorem 1.1]{deligneesnault}. The fact that it is closed under $\tau$-companions follows from the $L$-function argument in \cite[Proof of Theorem 1.1]{egintegral} (see also \cite[Proposition 7.4]{eg}). In greater detail: irreducibility is preserved under $\tau$-companions by the usual $L$-function argument, which implies that geometric irreducibility is as well as the $\tau$-companions relation is compatible with extension of base field. Similarly,  $\tau$-companions preserves the rank of cohomology (again via $L$-functions), hence sends cohomologically rigid local systems to cohomologically rigid local systems.
\end{proof}

\begin{proof}[Proof of Theorem~\ref{thm:main}(1)]
As in the statement of Theorem~\ref{thm:main}, suppose we have a spreading out $(\X, \mc{L}, S, \iota)$, and a closed point $s\in S$. Let $L_s\defeq \mc{L}|_{\X_s}$. Note that $L_s$ has algebraic determinant and is absolutely irreducible, hence by a theorem of Deligne \cite[Th\'eor\`eme 3.1]{delignefinitude}, the field $F_s$ generated by Frobenius traces is a number subfield of $E_{\ell}$. 

To prove the claim, it suffices to show that the number of $\ell$-companions of $L_s$ is bounded by an integer $A=A(X, n, d)$. Indeed, the companions are in bijection with embeddings $F_s\hookrightarrow \Qlbar$.

Let $M$ be a $\ell$-companion of $L_s$. By Lemma~\ref{lemma:companionclosed}, $M|_{\X_{\bar{s}}}$ is cohomologically rigid; moreover, Theorem \ref{thm:cusps_compatible} further implies that $M$ is an element  $\cohrig{\X_{\bar{s}}}{E_{\ell}}$, where $E_{\ell}$ is the field from the statement of Lemma~\ref{lemma:biject}.  By Lemma~\ref{lemma:biject}, we have   $|\cohrig{\X_{\bar{s}}}{E_{\ell}}|\leq |\cohrig{X}{E_{\ell}}|$, and the latter depends only on $X, n, d,$ and $v$.  Finally, we claim that the fibers of the map 
\[
\{\ell\text{-companions of}\  L_s \}\rightarrow \cohrig{\X_{\bar{s}}}{E_{\ell}}
\]
have size at most $nd$; this finishes the proof since we can simply take $A\defeq nd|\cohrig{X}{E_{\ell}}|$. It remains to prove the claim.  Indeed, suppose $M_1, M_2$ are companions of $L_s$ which are isomorphic upon restriction to $\X_{\bar{s}}$. Then $M_1\simeq M_2\otimes \psi$ for some character $\psi$ of $G_{\kappa(s)}$; since $M_1, M_2$ both have determinants of order $d$, we have $\psi^{\otimes nd}$, which gives the desired claim.
\end{proof}




\section{Proof of Theorem~\ref{thm:main} away from $\text{char}(s)$}\label{section:unramified} 
\subsection{} 

As before, let $X/\mb{C}$ be a smooth variety with good compactification $\bar{X}$.
Let $\rho_i$ be a list of the elements of  $\cohrig{X^{an}}{\C}$. There exists a number field $E$, a set of places $\Sigma$ of $E$, such that each  $\rho_i$ comes from a local system of projective $\mc{O}_{E, \Sigma}$-modules\footnote{where $\mc{O}_{E, \Sigma}$ denotes the ring of $\Sigma$-integers of $E$. In fact, by \cite{egintegral} we may take $\Sigma$ to be empty} of rank $n$. 

\begin{lem}\label{lemma:irred}
There exists $N=N(X, d, n)$ such that for all $\ell >N$ and every $i$,  $\rho_i$ is absolutely irreducible mod $\lambda$, where $\lambda$ is a prime of $\mc{O}_{E, \Sigma}$ lying above $\ell$.
\end{lem}

\begin{proof}
    It suffices to prove the statement for a fixed $\rho=\rho_i$. Let us write $\rho: \pi_1(X^{an}) \rightarrow \GL(M)$ for $M$ a projective module $\mc{O}_{\Sigma}$-module, which we may do by \cite{egintegral}. Suppose for the sake of contradiction that for infinitely many $\lambda$'s, $M/\lambda M\otimes \bar{\mb{F}}_{\lambda}$ has a  $\pi_1(X^{an})$-invariant subspace of rank $i$, with $0<i<n$. We fix generators $g_1, \cdots , g_m$ of $\pi_1(X^{an})$. 

    For any integer $1\leq i\leq n-1$, consider the $\mc{O}_{E, \Sigma}$-scheme $Z\subset \Gr(i, M)$ defined by $Z\defeq \cap_j\Gr(i, M)^{g_j}$, where $\Gr(i, M)^{g_j}$ is the (scheme-theoretic) fixed points of $g_i$, and we are taking the scheme theoretic intersection.

  Then $Z$ is a closed subscheme of $\Gr(i, M)$, which is of finite type over $\Spec(\mb Z)$ and has closed points in arbitrarily large characteristics, and therefore has a characteristic zero point, contradicting the absolute irreducibility of $M$. 
\end{proof}
The following follows straightforwardly from the fact that $\pi_1(\Xan)\rightarrow \piET X$ has dense image.
\begin{prop}\label{prop:containedinoe}
Let $X$ be a smooth variety over $\mb{C}$, and  $\rho: \piET{X}\rightarrow \GL_n(\Qlbar)$  a continuous representation. Suppose that the composition 
\[
\pi_1(\Xan)\rightarrow \piET{X}\xrightarrow{\rho}  \GL_n(\Qlbar) 
\]
has image in $\GL_n(\mc{O}_{E_{\ell}})$, where $\mc{O}_{E_{\ell}}$ is the ring of integers of an  $\ell$-adic local field $E_{\ell}$. Then the same is true for $\rho$.
\end{prop}

Let $Y/\mb{F}_q$ be a smooth variety, with $\bar{Y}$ the basechange to $\Fqbar$. We set $F\colon \bar{Y}\rightarrow \bar{Y}$ to be the $\Fqbar$-morphism induced via base change from the $q$-Frobenius on $Y$. There is an induced outer automorphism of $\piET{\bar{Y}}$, which we denote by $F^*$. This map may equivalently be described as follows: consider the homotopy exact sequence $0\rightarrow \piET{\bar{Y}}\rightarrow \piET{Y}\rightarrow G_{\mb F_{q}}\rightarrow 0$, where the latter is topologically generated by the $q$-Frobenius. For any element in $\piET{Y}$ lifting $q$-Frobenius, conjugation by this element gives the desired automorphism of $\piET{\bar{Y}}$. In particular, if $V$ is equipped with a continuous  action of $\piET{\bar Y}$, then there is another continuous action  of $\piET{\bar Y}$ on $V$, which we denote by $F^* V$, given by precomposing with this outer automorphism.

\begin{prop}[{\cite[\S~1.1.14]{weilii}}]\label{prop:delignedescend}
    Let $Y/\mb{F}_q$ be a smooth variety , and denote by $\bar{Y}$ its basechange to an algebraic closure $\Fqbar$. Let $E_{\ell}$ be an $\ell$-adic local field with $\ell\neq p$. Then there is an equivalence between the category of continuous, geometrically absolutely irreducible  $E_{\ell}$-representations of $\piET Y$, and the category of pairs $(V, \Phi)$, where $V$ is a continuous absolutely irreducible $E_{\ell}$-representation of $\piET{\bar{Y}}$    with torsion determinant, and  a $\piET{\bar{Y}}$-equivariant map
    \[
    \Phi: F^*V\simeq V,
    \]
    whose determinant is an $\ell$-adic  unit.
\end{prop}

\begin{proof}
    By \cite[\S~1.1.14]{weilii}, each $E_{\ell}$-representation $V_0$ of $\piET Y$ is in particular  a Weil sheaf on $Y$ and therefore gives a pair $(V, \Phi)$, with $V$ a representation of $\piET{\bar Y}$ and $\Phi\colon F^*V\simeq V$ a $\piET{\bar Y}$-equivariant map. By \cite[\S~1.3.4]{weilii} $\det V$ is torsion, say of order $d$. To see that $\Phi$ has determinant a unit, note that $(\det V_0)^{\otimes d}$  is a rank one  representation of $\Gal(\Fqbar/\mb{F}_q)$ where the action of a generator is given by $(\det(\Phi))^d,$ and hence a unit.
    
    Conversely, given $(V, \Phi)$ satisfying the stated hypotheses, we obtain a Weil sheaf $\mc{F}$ on $Y$, and it remains to show that it is an \'etale sheaf. Since $\mc{F}$ is assumed irreducible, by \cite[\S~1.3.14]{weilii} it suffices to show that  $L\defeq \det(\mc{F})$ is an \'etale sheaf. Let $\bar{\mc{F}}$ denote the \'etale sheaf on $\bar{Y}$ corresponding to $V$, whose determinant is of order $d$, say. Then $L^{\otimes d}$ is  a Weil sheaf pulled back from $\Spec(\mb{F}_q)$, and in fact an \'etale sheaf from the condition that $\det \Phi$ is a unit. Since the category of rank 1 \'etale sheaves on $\Spec(\mb{F}_q)$ is divisible, we may find a rank one \'etale sheaf $\psi$ on $\Spec(\mb{F}_q)$ such that $(L\otimes \psi)^{\otimes d}$ is trivial; in other words $L = (L\otimes \psi)\otimes \psi^{-1}$ is the tensor product of a Weil sheaf with torsion determinant with an \'etale sheaf on $\Spec(\mb{F}_q)$, and hence is itself an \'etale sheaf--see \cite[\S~0.4]{delignefinitude}.
\end{proof}

We record the following immediate consequence of the Brauer-Nesbitt theorem:
\begin{prop}\label{prop:descent}
    Let $G$ be a group equipped with an action on finite dimensional $E_{\ell}$-vector spaces  $V$ and $W$ which are moreover absolutely irreducible. Suppose $V$ and $W$ are isomorphic over $\Qlbar$; then the same is true of $V$ and $W$.  
\end{prop}

\begin{proof}[Proof of Theorem~\ref{thm:main}(2)]
     Suppose we have a spreading out $\mc{L}$ as in the statement of the theorem, corresponding to a representation $\rho'$ of $\piET{\X}$. Let $N_2$ be as in part (2) of Lemma~\ref{lemma:biject}. Let $\rho'_s$ be the restriction of $\rho'$ to  $\piET{\X_s}$, and for each prime number $m >N_2$ with $m\neq \text{char}(s)$, let $\rho_s$ be any $m$-companion of $\rho'_s$. To prove part (2) of Theorem~\ref{thm:main}, it suffices to  show that $\rho_s$ has coefficients in an unramified extension of $\mb{Q}_{m}$. By Lemma~\ref{lemma:biject} (2), we may assume that the restriction to $\X_{\bar{s}}$ takes the form $\rho_{\bar{s}}: \piET{\X_{\bar{s}}}\rightarrow \GL_n(E_{m})$, with $E_m$ being an unramified extension of $\mb{Q}_m$. Since $\rho_{\bar{s}}$ is integral, we will view it as a representation on a free $\mc{O}_{E_m}$-module $\mb{M}$, and let $V\defeq \mb{M}\otimes \mb{Q}$.




By Proposition~\ref{prop:descent},   there exists a  $\piET{\mf{X}_{\bar s}}$-equivariant map 
\[
\Phi: F^*V\simeq V,
\]
where $F$ denotes $q$-Frobenius.
Let $\pi\in E_{m}$ denote a uniformizer.
\begin{claim}\label{claim:lattice}
Assuming $m>N$ where $N=N(X, d, n)$ as in Lemma~\ref{lemma:irred},   there exists $M\in \Z$ such that $\pi^M\Phi$ has determinant a unit.
\end{claim}

Before giving the proof, we see how Claim~\ref{claim:lattice} finishes the proof of part (2) of Theorem~\ref{thm:main}. Indeed, the claim implies, by Proposition~\ref{prop:delignedescend}, that we may use $\pi^M \Phi$ to descend $V$ to a new representation  $\tau\colon \piET{\mf{X}_s}\rightarrow \GL_n(E_{m})$. Now, the geometric determinants of $\tau$ and $\rho_s$ are isomorphic, and therefore $\det(\tau)^{\otimes -1}\otimes \det(\rho_s)$ is a character of $\piET(\Spec(\mb{F}_q))$. Moreover, by our assumption that $\det(\rho_s)^{d}$ is an integral power of the cyclotomic character, we have  in fact 
\[
\det(\tau)^{\otimes -1}\otimes\det(\rho_s): \piET{\Spec(\mb{F}_q)}\rightarrow E_m'^{\times}.
\]
for some finite unramified extension $E_m'$ of $E_m$. Finally, taking $\psi$ to be an $n$-th root of this character, we have that
there exists an unramified extension $F_{m}/E'_{m}$, and a  character $\psi: \piET{\Spec(\mb{F}_q)}\rightarrow F_{m}^{\times}$, such that  \footnote{here we abuse notation and continue to denote the representation by  $\tau$ even after extending coefficients via $E_{m}\subset F_{m}$}  $\det(\tau\otimes \psi)\simeq \det(\rho_s)$. Here $F_m$ is unramified since it is obtained by taking an $n$-th root of $(\det(\tau)^{\otimes -1}\otimes\det(\rho_s))(\text{Frob}_q)$, which is an $m$-adic unit.

Now the two representations   $\tau\otimes \psi$ and $\rho_s$ of $\piET{\mf{X}_s}$ are geometrically isomorphic and have the same determinant. Therefore they differ by twisting by an order $n$ character, and hence $\rho_s$  has coefficients in $F_{m}(\zeta_n)$, which is unramified at $\ell$, as desired. 

This implies that  $\rho_s$  has coefficients in $F_{m}(\zeta_n)$, which is unramified at $m$, as desired.
\begin{proof}[Proof of Claim~\ref{claim:lattice}]
    Without loss of generality, by scaling $\Phi$ by powers of $\pi$, we may assume that $\pi \mb{M}\subsetneq \Phi(F^*\mb{M})\subset M$. Then the image 
    \[
    \Phi(F^*\mb{M})/\pi \mb{M} \subset \mb{M}/\pi \mb{M}
    \]
    is stable under the action of $\piET{\mf{X}_{\bar{s}}}$. By Lemma~\ref{lemma:irred}, $\Phi(F^*\mb{M})/ \pi \mb{M}$ must be the entirety of $\mb{M}/\pi \mb{M}$ or zero, and the   assumption $\pi \mb{M}\subsetneq \Phi(F^*\mb{M})$ in fact implies that we must be in the former case. Therefore $\Phi$ is an isomorphism, and hence its determinant is a unit, as claimed.
\end{proof}

\end{proof}
\section{$F$-isocrystals and proof of Theorem~\ref{thm:main}}\label{section:fcrystals}
In this section, we address the $\text{char}(s)$ part of Theorem \ref{thm:main}. For any smooth scheme $Y/k$ over a perfect field $k$, let $\fisocd{Y}$ denote the category of overconvergent$F$-isocrystals on $X$, which is  $\mb{Q}_p$-linear. For any extension $K/\mb{Q}_p$, let $\fisocd{X}_K$ denote the basechanged category, which is $K$-linear: for a detailed discussion we refer the reader to \cite[\S~3]{krishna}.

    This section was initially written under the assumption that $X/\C$ is projective, following \cite{eg}. After the first draft of this article was made public, Esnault-Groechnig released \cite{eg2} (born in part from \cite[Lecture 8]{elec}), which we use in lieu of \cite{eg} to make our results work in the quasi-projective case.

\subsection{}We summarize one of the results of \cite{eg2}. Let $k$ be a finite field of odd characteristic, with ring of Witt vectors $W=W(k)$ and set $K:=W[1/p]$. Let $\X/W$ be a smooth morphism with good compactification $\bar{\X}/W$. We say a flat connection $(\mc V,\nabla)$ on $\X$ has \emph{quasi-unipotent monodromy at $\infty$} if the residues, $\lambda_{ij}\in \Zpbar$, are in fact in $\Z_{(p)}\subset \Q\subset \Qpbar$, the localization of $\Z$ away from the prime ideal $(p)$.

\begin{thm}\label{thm:EG_p-adic}[Esnault-Groechnig]
Let $(\mc V,\nabla)$ be a cohomologically rigid relative log flat connection on $(\X,\bar{\X})/W$ of rank $n$ with $d$-torsion determinant and quasi-unipotent monodromy at $\infty$. Then $(\mc V,\nabla)|_{\X_{K}}$ admits the structure of an overconvergent $F^f$-isocrystal.

Pick an embedding $K\hookrightarrow \C$ and suppose that the generalized eigenvalues around the boundary are contained in $\mu_v\subset \C^{\times}$. Then the integer $f$ may be taken to be $|\cohrig{\X_{\C}}{\C}|!$, i.e., the size of the permutation group of the set $\cohrig{\X_{\C}}{\C}$.
\end{thm}

\begin{proof}
    The fact that $(\mc V,\nabla)$ underlies an overconvergent $F^f$ isocrystal for some $f$ directly follows from \cite[Theorem 1.2]{eg2}. The bound on $f$ follows from the techniques of \cite{eg2}, but we say a few words. First, suppose that $(\mc V, \nabla)$ has unipotent monodromy around the boundary divisor (i.e., has nilpotent residues); then $(\mc V, \nabla)$ is a crystal on the log crystalline site by \cite[Theorem 1.2]{eg2}.
    Then, the Frobenius pullback $F^*(\mc V, \nabla)$ is again a crystal on the log crystalline site with nilpotent residues (see 3.25 of \emph{loc. cit.}, which is moreover cohomologically rigid (as Frobenius pullback will not change rational cohomology).\footnote{In fact, Frobenius pullback sends rigid flat connections to rigid flat connections by Corollary 4.7 of \emph{loc. cit.}} In particular, Frobenius pullback permutes the finitely many isomorphism classes of cohomologically rigid log flat connections with nilpotent residues on $(\X, \bar{\X})/W$ with the required rank and determinant condition. This number is no greater than $\cohrig{\X_{\C}}{\C}$ by Deligne's Riemann-Hilbert correspondence..
    
    In the quasi-unipotent monodromy case (with non-nilpotent residues), one may simply imitate the above reasoning on a root stack, as in Secion 3.6 of \emph{loc. cit.}
    
\end{proof}
\begin{rmk}We emphasize that, in the context of Theorem \ref{thm:EG_p-adic}, the choice of the $F^f$ structure is \emph{not} canonical: for instance, we can scale it by any number in $\Zp^{\times}$. As a result, there is \emph{no reason} to suppose that the output $F^f$-isocrystal is ``algebraic'' (in the sense of having algebraic determinant, or having traces of Frobenius elements being algebraic numbers).
\end{rmk}

The following fact relating $F^f$-isocrystals to $F$-isocrystals with coefficients is well-known.
\begin{fact}
    Let $Y/\mathbb F_{q}$ be a smooth variety and let $f\geq 1$. If $\mb F_{p^f}\subset \mb{F}_q$, then there is an natural equivalence of categories between $F^f$-isocrystals on $Y$ and $F$-isocrystals on Y with coefficients in $\mb Q_{p^f}$:
    $$\ffisoc{Y}\rightarrow \fisoc{Y}_{\mb Q_{p^f}}$$
\end{fact}
\begin{rmk}
    This fact is false without the hypothesis that $\mb F_{p^f}\subset \mb{F}_q$. For instance, it is easy to find counterexamples setting $Y=\Spec(\mb F_p)$.
\end{rmk}

To control the ramification of the field of Frobenius traces at $\text{char}(s)$, we need a few basic results on $F$-isocrystals with unramified coefficients. For simplicity, these are tailored to our application. Before stating these, we need a definition.
\begin{defn}\label{def:r_extension}
    Let $Y/k$ be a smooth variety over a finite field of characteristic $p$, let $r\geq 1$, and let $K$ be a $p$-adic local field. Let $\mc E\in \fisoc{Y}_K$.  We set:
    \begin{itemize}
    \item $k_r/k$ to be the smallest field extension that contains $\mb{F}_{p^r}$; 
    \item $Y_r=Y\times_{\Spec(k)}\Spec(k_r)$ and 
    \item $\mc E_r$ to be the \emph{restriction} of $\mc E$ to $Y_r$.
    \end{itemize}
\end{defn}
\begin{prop}\label{prop:mth_root_fisoc}
    Let $k=\mb F_{p^a}$, let $\psi\in \fisoc{k}_{\mb Q_{p^f}}$ be a rank 1 object, and let $m \geq 1$ be an integer. Then there are exactly $m$ isomorphism classes of rank 1 objects $\chi_i\in \fisoc{k}_{\Qpbar}$ with the property that $\chi_i^{\otimes m}\cong \psi$.
    
    Suppose further that $f|a$ and that $1\leq m< p$ and set $M:=fm$. Then for each $1\leq i\leq m$, there exists an extension $K_i/\mb Q_{p^a}$ of degree at most $m$, such that the object $\chi_i$ satisfies the following properties:
    \begin{itemize}
        \item $\chi_i\in \fisoc{k}_{K_i}$; and
        \item the restriction $(\chi_i)_M$ of $\chi_i$ to $k_M$ has Frobenius trace in an unramified extension of $\Qp$.
    \end{itemize}
    
\end{prop}
In other words, for each  $m^{\text{th}}$-root $\chi$ of $\psi$, if $k\supset \mb F_{p^{fm}}$, then the Frobenius trace of $\chi$ is unramified over $\Qp$. It may be helpful to have in mind the example of $\Qpbar(-1/2)\in \fisoc{\mb F_p}_{\Qpbar}$; when restricted to $\mb F_{p^a}$, where $a$ is even, the Frobenius trace is in $\Qp$.
\begin{proof}
First of all, $\psi$ corresponds to a pair $(M,F)$, where $M$ is a free $\mb Q_{p^a}\otimes \mb Q_{p^f}$-module of rank 1 and $F$ is a $\sigma\otimes \text{id}$-linear map by \cite[Proposition 5.2]{krishna}. Then it follows from \cite[Proposition 5.8]{krishna} and the surrounding discussion that  $\psi$ is determined by the element $\lambda:=\text{Tr}(F^a)\in \mb Q_{p^f}^{\times}$. (In other words, $F^a$ is a \emph{linear} map $M\rightarrow M$ between free $\mb Q_{p^a}\otimes \mb Q_{p^f}$-modules of rank 1, so is given by multiplication by an element of $\mb Q_{p^a}\otimes \mb Q_{p^f}$. One proves this element in fact lives in the subring $\mb Q_p\otimes \mb Q_{p^f}\cong \mb Q_{p^f}$.)

Fix an $m$-th root $\lambda^{1/m}\in \Qpbar$ of $\lambda$. Consider the cyclic algebra map $\Qpbar\rightarrow \Q_{p^a}\otimes\Qpbar$ (with cyclic structure given by $\sigma\otimes \text{id}$). The resulting norm map is clearly surjective: $\Q_{p^a}\otimes\Qpbar\cong \prod \Qpbar$, where the cyclic action on the right hand side is given by shifting. Therefore $\lambda^{1/m}$ is in the image of the norm map, and hence there exists a rank 1 object $\chi\in \fisoc{k}_{\Qpbar}$ with $F^a$ being multiplication by $\lambda^{1/m}$,   by the discussion in between Proposition 5.7 and Proposition 5.8 of \cite{krishna}; therefore $\chi^{\otimes m}\cong \psi$. Moreover, any $m^{\text{th}}$-root of $\psi$ in $\fisoc{k}_{\Qpbar}$ is given by this construction.

Note that $\lambda$ has the following special property with respect to the cyclic algebra map $\mb Q_{p^f}\rightarrow \mb Q_{p^a}\otimes \mb Q_{p^f}$ (with cyclic structure given by $\sigma\otimes \text{id}$): $\lambda$ lies in the image of the resulting norm map $(\mb Q_{p^a}\otimes \mb Q_{p^f})^{\times}\rightarrow \mb Q_{p^f}^{\times}$. 

Now, let us assume that $f|a$. Since $\mb Q_{p^a}\otimes \mb Q_{p^f}\simeq \prod \mb{Q}_{p^a}$ and the norm map from the latter to $\mb{Q}_{p^f}$ is the product of the individual norm maps,  we have that $\lambda$ lies in the image of $\text{Nm}: \mb{Q}_{p^a}\rightarrow \mb{Q}_{p^f}$. Therefore $\frac{a}{f}\mid v_p(\lambda)$, by the explicit description of the image of the norm map of the extension $\mb Q_{p^a}/\mb Q_{p^f}$.




Let $K:=\mb{Q}_{p^a}(\lambda^{1/m})\subset \Qpbar$. We now claim  that $\lambda^{1/m}\in K^*$ is a norm with respect to the cyclic algebra map $K\rightarrow \mb Q_{p^a}\otimes K$, where the cyclic structure is again given by $\sigma\otimes \text{id}$.

Indeed, $$\mb Q_{p^a}\otimes K\cong \prod K,$$  and the norm map on the left hand side transports to the product on the right hand side. Therefore, the above norm map is surjective. It follows from \cite[Corollary 5.9]{krishna} that $\chi$ descends to an object of  $\fisoc{k}_{K}$.


Let $M=fm$. We now claim that given our choice of $\lambda^{1/m}\in \Qpbar$ and the induced $\chi$, the Frobenius trace field of $\chi_{M}$ (the restriction of $\chi$ to $\fisoc{k_{M}}_K$) has Frobenius trace in an unramified extension of $\mb{Q}_{p}$. 

To prove this, note that the construction of $\chi$ as an $m^{\text{th}}$-root of $\psi$ may as just as well have been done on $k_M$. Therefore, to prove the above, we may as well assume that $M|a$; since $\chi$ has coefficients in $K$, it suffices to show that $K$ is unramified over $\mb{Q}_p$. In this case, note that $m|v_p(\lambda)$ (normalized with $v_p(p)=1$): indeed, $\lambda$ is in the image of the norm map $\text{Nm}\colon \mb Q_{p^a}^{\times}\rightarrow \mb Q_{p^f}^{\times}$, with $mf|a$. As $m<p$, it follows that $K/\mb Q_{p^a}$ is an unramified extension, as required.
\end{proof}

The following lemma uses Proposition \ref{prop:mth_root_fisoc} to deal with higher rank $F$-isocrystals . We recommend first parsing the following lemma when $d=1$ and $g = 0$ to understand what it is saying.
\begin{lem}\label{lem:twist_p-adic}
    Let $d, f\geq 1$, let $\mb F_{p^f}\subset k$ be a finite field, let $Y/ k$ be a smooth variety and let $$\mathcal E\in \fisoc{Y}_{\mb{Q}_{p^f}}$$ be an object of rank $n$. Suppose that the underlying (convergent) isocrystal of $\mc E$ has determinant of order dividing $d$, and furthermore that $nd< p$. Let $g\in \Z$ and set $$S_{g,d}:=
    \{\mc E'\cong \mc E\otimes \chi \in \fisoc{Y}_{\Qpbar}|\ \chi\in \fisoc{k}_{\Qpbar}^{\text{rank 1}},\ (\det(\mc{E}')\otimes \Qpbar(-g))^{\otimes d}\cong \Qpbar\},$$ i.e., $S_{g,d}$ is the set of isomorphism classes of twists of $\mc E$ whose determinant is isomorphic to $\Qpbar(g)$ tensor a rank 1 object of order $d$.

    Then $|S_{g, d}|=dn$, and for every $\mc E'\in S_{g, d}$, the restriction $(\mc E')_{fdn}$ has Frobenius traces in an unramified extension of $\Qp$.


\end{lem}
\begin{proof}
    First of all, we claim that $\det(\mc E)^{\otimes d}\in \fisoc{Y}_{\mb{Q}_{p^f}}$ is the pull back of an object of $\fisoc{k}_{\mb{Q}_{p^f}}$. This follows from the fact that the  convergent isocrystal underlying $\det(\mc{E})^{\otimes d}$ is trivial, and the fact that an $F$-isocrystal with trivial underlying isocrystal is pulled back from the base.

    We will first prove the lemma for $g = 0$. Let us first construct $nd$ such rank 1 objects $\chi$. Set $\psi:=\det(\mc E)^{-d}$ and $m=dn$. Let $\chi$ be one of the $dn$ outputs of Proposition \ref{prop:mth_root_fisoc}. Note that $(\mc E\otimes \chi)\in \fisoc{Y}_K$ has determinant of order $d$; indeed, $\det(\mc E\otimes \chi)^d = \det(\mc E)\otimes \chi^{dn}$. Then Proposition \ref{prop:mth_root_fisoc} implies that $(\mc E\otimes \chi)_{fdn}$ has Frobenius traces in a field unramified over $\Qp$.



    On the other hand,  if  $\mc F = \mc{E}\otimes \chi \in \fisoc{Y}_{\Qpbar}$ is a twist with  order $d$ determinant, then  $\chi^{-dn}=(\det{\mc E})^{d}$. Hence Proposition \ref{prop:mth_root_fisoc} implies that the above construction exhausts $S_{g, d}$.

    To prove the lemma for arbitrary $g\in \Z$, pick any $p^{-1/n}\in \Qpbar$ and consider the following rank 1 object $(M,F)$ in $\fisoc{\mb F_p}_{\Qpbar}$, which we call $\Qpbar(1/n)$: $M$ is a free rank 1 vector space over $\Qpbar$ and $F$ is given by multiplication by $p^{-1/n}$. We may define $\Qpbar(g/n):=\Qpbar(1/n)^{\otimes g}$. Then there is a bijection:

    $$S_{0,d}\xrightarrow[]{\otimes \Qpbar(g/n)} S_{g, d}.$$
    By the same argument as in Proposition \ref{prop:mth_root_fisoc}, it follows that the Frobenius trace of the restriction $(\Qpbar(g/n))_{n}\in \fisoc{\mb F_{p^{n}}}_{\Qpbar}$ is in $\Qp$ (in fact, it is exactly $1/p^g$), and the lemma follows.
\end{proof}

Putting the above together, we have the following theorem, which provides a $p$-adic bound on the field of Frobenius traces in the theorem of Esnault-Groechenig. Loosely, it says that for any algebraic $F$-isocrystal that is an output of their theorem, after restricting to a slightly bigger base field, the number field generated by Frobenius traces is unramified over $p$. 

\begin{defn}\label{def:MdR}

Let $X/\C$ be a smooth variety with good compactification $\bar{X}$. Fix integers $n,d,v\geq 1$. Let $(\X, \bar{\X})/S$ be a spreading out, such that $d$ and $v$ are invertible in $S$.

Let $\MdR{X}{\C}$ denote Simpson's moduli space uniformly corepresenting log flat connections on $(X,\bar X)$ of rank $n$, with determinant of order $d$, and whose residues around boundary divisors are rational with denominators dividing $v$ \cite[Theorem 3.8]{simpsonmodI}. Let $\cohrigdR{X}{\C}$ denote the set of smooth isolated points of $\MdR{X}{\C}$.

Let $\MdR{\X}{S}$ denote Langer's quasi-projective space uniformly corepresenting of (Gieseker) stable log flat connections on $(\X, \bar{\X})/S$ of rank $n$, with vanishing Chern classes, with determinant of order $d$ and whose residues around boundary divisors are rational with denominators dividing $v$ \cite[Theorem 1.1]{langersemistable}. Let $\cohrigdR{\X}{S}\subset \MdR{\X}{S}$ denote the maximal open subscheme such that the structure morphism is quasi-finite and smooth (c.f. \cite[Definition 3.2]{eg}).
\end{defn}
\begin{rmk}
    By \cite[Theorem 1.1]{langersemistable}, for any geometric point $\bar{s}\hookrightarrow S$, the set $\MdR{\X_{\bar s}}{\bar s}$ is precisely the set of isomorphism classes of stable log flat connections on $(\X_{\bar s}, \bar{\X}_{\bar s})$ of rank $n$, vanishing Chern classes, determinant of order $d$, and rational residues with denominators dividing $v$. 
\end{rmk} 

\begin{setup}\label{setup:global_dR}
    Fix $X,\bar{X}, n,d,v$ as in the introduction. Let $\MdR{X}{\C}$ be as above and let $(\X, \bar{\X})\rightarrow S$ be a spreading out such that
    \begin{itemize}
        \item $\Z[\zeta_{nd}]\subset\Gamma(S,\mc O_S)$
        \item the morphism $\cohrigdR{\X}{S}\rightarrow S$ is isomorphic to $\coprod S\rightarrow S$\footnote{This concretely implies the following: every $(V_i,\nabla_i)$ in $\cohrigdR{X}{\C}$ has a canonical spreading-out to an $S$-relative flat connection $(\mc V_i, \nabla_i)$ on $\mf X$ whose fibers are all stable and cohomologically rigid.}; and
        \item every element of $\cohrig{X}{\Qlbar}$ spreads out to $\X$ with $d$-torsion determinant (i.e., satisfies the conclusion of Lemma~\ref{lem:spreading_out}).
    \end{itemize}
\end{setup}

\begin{rmk}\label{rmk:global_dR_etale}
    In the context of Setup \ref{setup:global_dR}, the degree of $\cohrigdR{\X}{S}\rightarrow S$ is precisely $|\cohrigdR{X}{\C}|$, which by Deligne's Riemann-Hilbert correspondence is the same as $|\cohrig{X}{\Qlbar}|$.
\end{rmk}
\begin{rmk}
    We comment on why Setup~\ref{setup:global_dR} is possible.
    
    The \'etale part:
    \begin{itemize}
        \item for any spreading out $(\X,\bar{\X})/S$, the scheme $\cohrigdR{\X}{S}\rightarrow S$ is quasi-finite and smooth. Therefore, by Grothendieck's version of Zariski's main theorem, after shrinking $S$ we may assume that the map is finite and smooth, i.e., finite and \'etale. Then replacing $S$ by an appropriate finite \'etale cover will split this cover.
    \end{itemize}
    The de Rham part:\footnote{In the projective case, see \cite[Proposition 4.10]{eg} or \cite[8.5]{elec}.} 
    \begin{itemize}
        
        \item there are only finitely many such flat connections (which follows from Deligne's Riemann-Hilbert correspondence);
        \item stable-ness is an open condition; and
        \item relative (logarithmic) de Rham cohomology satisfies base change when restricted to a open dense subset of $S$ (see e.g. \cite[Theorem 8.0, Corollary 8.4]{katznilp}).
    \end{itemize}
\end{rmk}

The following lemma is a generalization of \cite[Theorem 7.3(1)]{eg} to the quasi-projective case. We claim no originality for this; one simply combines \cite{eg2} with the exact technique of \cite[Setion 7]{eg}
\begin{lem}\label{lem:all_companions_exist}
    Setup as in Setup \ref{setup:global_dR}. Let $(V,\nabla)$ be in $\cohrigdR{X}{\C}$, with canonical spreading out $(\mc V, \nabla)$ over $(\X,\bar{\X})/S$. Let $p>dn$ be a prime.

    For any point $\tilde{s}\colon \Spec(W(\mb F_q))\rightarrow S$, where $q$ is a power of $p$, $(\mc V, \nabla)$ gives rise (non-canonically) to an object  $\mc E\in \fisocd{\X_s}_{\Qpbar}$ with algebraic determinant.

    Then for any $\tau\in \text{Aut}(\Qpbar)$, the $\tau$-companion to $\mc E$ exists, and moreover is constructed from the same process as $\mc E$, i.e., the underlying isocrystal of $^\tau \mc E$ globalizes to a flat connection $(\mc W, \nabla)$ on $(\X, \bar{\X})/S$ corresponding to a section of $$\cohrigdR{\X}{S}\rightarrow S$$.
\end{lem}
\begin{proof}
Let $(\mc V_i,\nabla_i)$ be the canonical spreading-outs of the $$(V_i,\nabla_i)\in\cohrigdR{X}{\C}$$ to $(\X,\bar{\X})/S$. Assume that our given flat connection is $(V_1,\nabla_1)$. Note that by construction of $(\X,\bar{\X}/S)$, for all closed points $s$, the $(\mc V_i,\nabla_i)|_{\X_{\bar s}}$ are pairwise non-isomorphic log flat connections on $(\X,\bar{\X})_{\bar s}$.

For each $i$, let $\{\mc E_{i,j}\}$ be the collection of Frobenius structures on the $p$-adic completion $(\mc V_i,\nabla_i)|_{\X_{\tilde s}}$  (so each $\mc E_{i,j}\in \fisocd{\X}_{\Qpbar}$) with order $d$ determinant; for a given $i$, there are exactly $dn$ of these. (Any two are related by twisting by a rank 1 object of $\fisoc{\mb F_q}_{\Qpbar}$ of order $dn$, of which there are exactly $dn$.) We will argue that $\{\mc E_{i,j}\}$ form a complete set of $p$-adic companions, i.e., for any $\tau\in \text{Aut}(\Qpbar)$, the $^{\tau}\mc E_{i,j}\cong \mc E_{i',j'}$ for some $(i',j')$.

Fix $\iota\colon \Qpbar\rightarrow \Qlbar$. Then $\{^{\iota}\mc E_{i,j}\}$ is a set of mutually non-isomorphic $\Qlbar$-sheaves on $\X_{s}$, and for fixed $i$, the set $\{^{\iota}\mc E_{i,j}\}$ consists of the $dn$-torsion twists of a fixed $\Qlbar$-sheaf. (This in particular means that for fixed $i$, the above sheaves are all isomorphic over $\X_{\bar s}$. Moreover, we know that the relation of $\iota$-companions preserves the property of being cohomologically rigid (again using the $L$-function technique of  \cite[Proof of Theorem 1.1, Lemma 3.4]{egintegral}, with key input Lemma 3.4 of \emph{loc. cit.})

By the choices of $\X/S$, the set $\{^{\iota}\mc E_{i,j}\}$ is closed under the relation of $\ell$-adic companions: companions preserves being cohomologically rigid, and our choice of spreading out implied that every cohomologically rigid $\Qlbar$-sheaf on $\X_{\bar s}$ with the prescribed numerics actually globalizes to $\X$. It follows that $\{\mc E_{i,j}\}$ is closed under the relation of $p$-adic companions, as desired.

\end{proof}

\begin{thm}\label{thm:p-adic_traces_unr}
    Setup as in Setup~\ref{setup:global_dR}. Let $(V,\nabla)$ be a cohomologically rigid de Rham bundle of rank $n$, determinant of order $d$, and rational residues around boundary components with denominators dividing $v$ on $X$, with canonical spreading out to $(\mc V,\nabla)$. Let $p>dn$.
    
    Then for any point $\tilde{s}\colon \Spec(W(\mathbb F_q))\rightarrow S$, where $\mathbb F_q\supset \mathbb F_{p^{f}}$, and for any $\mc E\in \ffisoc{\mf X_{ s}}\cong \fisoc{\mf X_{s}}_{\mb Q_{p^f}}$ such that the underlying convergent isocrystal of $\mc E$ admits $(\mc V, \nabla)$ as a lattice, and for any twist $\mc E'\in \fisoc{\mf X_s}_{\Qpbar}$ of $\mc E$ with $(\det(\mc E')\otimes \Qpbar(g))^{\otimes d}\cong \Qpbar$ for some $g\in \Z$, 
    the number field generated by Frobenius traces of $\mc E'_{fdn}$ is unramified over $p$. 
\end{thm}
Note that Theorem \ref{thm:EG_p-adic} implies that the theorem is not vacuous.
\begin{proof}
It follows from Theorem \ref{thm:EG_p-adic} and Lemma \ref{lem:twist_p-adic} that for any such $\mc E'$, the field generated by Frobenius traces of $\mc E'_{fdn}$, as a subfield of $\Qpbar$, is contained in an unramified extension of $\Qp$. This in particular means that the \emph{number field} generated by Frobenius traces has at least one $p$-adic place $\lambda$ that is unramified. It is our goal to that the same holds for every other $p$-adic place; equivalently, that for every $\tau\in \text{Aut}(\Qpbar)$, the field generated by Frobenius of traces of $^{\tau}\mc E'\in \fisocd{X}_{\Qpbar}$, considered naturally as a subfield of $\Qpbar$, is contained in an unramified extension of $\Qp$. This follows immediately from Lemma \ref{lem:all_companions_exist}.

Therefore the number field generated by Frobenius traces of $\mc E'_{fdn}$ is unramified over $p$.
\end{proof}
\begin{cor}\label{cor:fcrystals}
    Let $\mf X \rightarrow S$ be a spreading out satisfying both Setup \ref{setup:global_dR} and also the conclusion in Lemma \ref{lem:spreading_out}. Let $p>nd$.

    Then there exists an $f\geq 1$ with the following property. Let $\mb L$ be a spreading out of $L\in \cohrig{X}{\Qlbar}$ $\Qlbar$-local system. Let $q=p^a$, with $fdn|a$. Then for any point
    $$s\colon \Spec(\mb F_q)\rightarrow S,$$
    the Frobenius trace field of $\mb L|_{\mf X_s}$ is unramified over $p$. In fact, $f$ make be taken to be $|\cohrig{X}{\C}|!$.

\end{cor}

\subsection{Proof of Theorem~\ref{thm:main}}
\begin{proof}
Parts (1) and (2)  of Theorem~\ref{thm:main} were proven in Sections~\ref{section:degree} and \ref{section:unramified} respectively, so it remains to prove part (3). 

Let $(\X, \bar{\X}, \mc{L}, S, \iota)$ be as in the statement of Theorem~\ref{thm:main}. 

By Corollary~\ref{cor:fcrystals}, there exists\footnote{more precisely, we can take a spreading out of $X$ to a $S'$ which lies over the spreading out to $S$ as well as the one in Setup~\ref{setup:global_dR}}  a finite type scheme $S'$, equipped with a map $S'\rightarrow S$ such that $(\X_{S'}, \mc{L}_{S'})$ satisfies  the conclusion of points (1), (2)   of Theorem~\ref{thm:main}, as well as the conclusion of Corollary~\ref{cor:fcrystals}. 

More precisely,  there exists  $\beta\in \mb{Z}_{\geq 1}$ such that for all $s: \Spec(\mb{F}_q)\rightarrow S'$ such that $\mb{F}_{p^{\alpha}}\subset \mb{F}_q$, $\mc{L}_{\X_{S'}}|_{\X_{S'}\otimes s}$  has trace field is
\begin{itemize}
    \item of degree $\leq N=N(X, n, d)$, and 
    \item unramified at any prime $m>\ell_0=\ell_0(\X, n,d)$.
\end{itemize}
Since the set of such number fields is finite \cite[Ch. II Theorem 2.13]{neukirch2013algebraic}, the local system $\mc{L}_{S'}$ satisfies the conclusion of point (3) of Theorem~\ref{thm:main}. Taking an \'etale map $T\rightarrow  S$ such that the induced morphism $S'_T \defeq S'\times_{S} T\rightarrow T$ has a section, we may conclude.
\end{proof}

\section{Corollaries of the main theorem}\label{section:adjointquestion}

\subsection{The adjoint representation}\label{section:adjoint}
\begin{prop}\label{prop:adjoint}
    For $L\in \cohrig{X}{\Qlbar}$, there exists a spreading out $(\X, S, \mc{L}, \iota)$ such that $\ad(\mc{L})$ has bounded Frobenius traces.
\end{prop}
The following is the key lemma.

\begin{lem}\label{lem:trace_adjoint}
    Let $Y/\mb F_q$ be a smooth variety  and let $L$ be an irreducible $\Qlbar$ local system with finite order determinant on $Y$. Then the Frobenius trace field of $\ad(L)$ is a subfield of the stable trace field of $L$.
\end{lem}
\begin{proof}

Set $E'\subset E\subset \Qlbar$ to be the stable trace field inside of the Frobenius trace field of $L$ and $K'\subset K\subset \Qlbar$ to be the stable trace field inside of the Frobenius trace field of $\ad(L)$. Of course, $K\subset E$ and $K'\subset E'$, and our goal is to prove that $K\subset E'$.

If $E=E'$, then we are done, so we may assume that $E'$ strictly contains $E$; we may suppose that $E'$ is in fact the Frobenius trace field of the restriction of $L$ to $Y':=Y\times_{\mb F_q}\mb F_{q'}$ for some power $q'$ of $q$. Furthermore, we may assume that there is $\lambda\in K\subset E$ which is not in $E'$.

Then there exists $\tau\in \text{Aut}(\Qlbar)$ with the following properties:
\begin{enumerate}
    \item $^{\tau}L$ is not isomorphic to $L$ and $^{\tau}\ad(L)$ is not isomorphic to $\ad(L)$
    \item $^{\tau}L$ and $L$ are isomorphic on $Y'$, and similarly with $\ad(L)$.
\end{enumerate}
Note that the operation of $\tau$-companions commutes with tensor product and dual, so $^{\tau}\ad(L)\cong \ad(^{\tau}L)$.

On the other hand, as $^{\tau}L$ and $L$ are isomorphic over $Y'$, it follows that there exists a character $\psi$ such that $^{\tau}L\cong L\otimes \psi$. It follows that $\ad(^{\tau}L)\cong \ad(L)$, 
which is a contradiction. Therefore, such a $\lambda$ cannot exist and $K\subset E'$, as desired. 
\end{proof}

\begin{proof}[Proof of Proposition \ref{prop:adjoint}]
    By Theorem \ref{thm:main}, there exists a spreading out $(\X,S, \mc L, \iota)$ and a number field $E$ such that the stable trace field of $\mc L|_{\X_s}$ is contained in $E$. It follows from Lemma \ref{lem:trace_adjoint} that the Frobenius trace field of $\ad(\mc L)$ is contained in $E$, as desired.
\end{proof}
\subsection{Boundedness for $L^{\oplus h}$}
The goal of this section is the prove the following.
\begin{thm}\label{thm:L^h_bounded}
     For $L\in \cohrig{X}{\Qlbar}$, there exists an integer $h$ and a  spreading out of $L^{\oplus h}$ with bounded Frobenius traces.
\end{thm}
We first require the following lemma, analogous to Lemma~\ref{lem:trace_adjoint}.
\begin{lem}\label{lem:L^h_stable}
    
    Let $Y/\mb F_q$ be a smooth variety and let $L$ be an irreducible $\Qlbar$ local system with determinant finite order. Let $y\in Y(\mb F_q)$. Then the Frobenius trace field of $L\otimes L_y^{\vee}$ is a subfield of the stable trace field of $L$.
\end{lem}
The idea to consider the tensor product $L\otimes L_y^{\vee}$ is essentially due to Kerz \cite[Lemma 6.2]{sasha}.
\begin{proof}
    The proof is analogous to Lemma~\ref{lem:trace_adjoint}. Let $E'\subset E\subset\Qlbar$ be the stable trace field inside of the Frobenius trace field of $L$. Then the trace field of $L\otimes L_y^{\vee}$ is a subfield $K\subset E$, and it is our goal to prove that $K\subset E'$.

    As before, we may assume that $E'\neq E$. Suppose for contradiction that $K\not\subset E'$. Take $\tau\in \text{Aut}(\Qlbar)$ that is the identity on $E'$ but non-trivial on $K$. 
    Our goal will be to prove that the $\tau$-companion to $L\otimes L_y^{\vee}$ is again isomorphic to $L\otimes L_y^{\vee}$, which will contradict the choice of $\tau$ and hence the assumption that $K\not\subset E'$.
    
    Note that $^{\tau}L\cong L\otimes \chi$, where $\chi$ is a rank 1 $\Qlbar$-local system on $\Spec(\mb F_q)$, as $\tau$ fixes $E'$ and $L$ is irreducible.
    Therefore, $^\tau(L\otimes L_y^{\vee})\cong L\otimes \chi \otimes L_y^{\vee}\otimes \chi_y^{\vee}$. As $\chi$ is pulled back from $\mb F_q$, it follows that $^{\tau}(L\otimes L_y^{\vee})\cong L\otimes L_y^{\vee}$, contradicting our assumption on $\tau$ as explained above. Therefore $K\subset E'$,  as desired. 
\end{proof}
\begin{proof}[Proof of Theorem~\ref{thm:L^h_bounded}]
    Set $h=n$, where $n$ is the rank of $L$. Let $(\mf X, \mc L, S, \iota)$ be a spreading out as in Theorem~\ref{thm:main}(\ref{mainthm:3}). After possibly replacing $S$ by a scheme \'etale over it, we may assume that the map $\mf X\rightarrow S$ has a section $s\colon S\rightarrow \mf X$. Consider $\mc L\otimes \mc L^{\vee}_{s}$ as a local system on $\mf X$. (Here, $\mc L^{\vee}_{s}$ is a local system on $S$, and we view it as a local system on $\X$ via pullback along the structure map.) This is clearly a spreading-out of $L^{\oplus n}$, so the task is to prove that the field generated by Frobenius traces is bounded. Theorem~\ref{thm:main} implies that  $\mc L$ has bounded stable Frobenius traces; combining with Lemma~\ref{lem:L^h_stable}, the result follows.
\end{proof}

\subsection{Rigidifying}\label{section:rigidify}

Recall that it cannot be true that the field of Frobenius traces of an \emph{arbitrary} spreading-out of a rigid local system is bounded; indeed, twisting \emph{does not} preserve this boundedness. Esnault suggested to us a formulation to rigidify the problem.
\begin{setup}\label{setup:rigidifying}
    Let $\mf X/S$ be a smooth scheme over an affine scheme $S=\Spec(A)$, where the structure map $\Z\rightarrow A$ is smooth and $A\subset \C$. Let $\mc L$ be a geometrically irreducible $\Qlbar$-local system on $\mf X$ with fractional cyclotomic determinant.
\end{setup}
\begin{thm}
    Notation as in Setup \ref{setup:rigidifying}. Suppose further that
    \begin{enumerate}
        \item $\mc L$ has bounded stable Frobenius trace field, i.e., there exists a number field $E_1\subset \Qlbar$ such that the stable Frobenius trace field is contained in $E_1$; and
        \item there exists a section $\sigma\colon S\rightarrow \mf X$ such that $\sigma^* \mc L$ has bounded coefficients of characteristic polynomials of Frobenius, i.e., there exists a number field $E_2$ such that for every closed point $s$ of $S$, the characteristic polynomial $P_s(\sigma^* \mc L, t)\in E_2[t]\subset \Qlbar[t]$. 
    \end{enumerate}
    Then there exists an \'etale map $S'\rightarrow S$ with the following property. Denote by $\mf X'$ the base-change of $\mf X\rightarrow S$ to $S'$. Then the Frobenius trace field of $\mc L|_{\mf X'}$ is contained in the compositum $E:=E_1\cdot E_2\subset \Qlbar$. In particular, $\mc L|_{\mf X'}$ has bounded Frobenius traces.
\end{thm}
\begin{proof}
    We first construct $S'\rightarrow S$ when $\ell\geq 3$. The local system $\sigma^*\mc L$ corresponds to a representation $\piET{S}\rightarrow \text{GL}_n(M)$, where $M$ is an $\ell$-adic local field. By compactness, we may (non-canonically) pick a stable lattice to obtain a representation $\piET{S}\rightarrow \text{GL}_n(\mc O_M)$. Then there is a cover $S'\rightarrow S$ corresponding to the open subgroup $\Gamma(\ell^n)\subset \text{GL}_n(\mc O_M)$ of matrices that are congruent to the identity mod $\ell^n$. (When $\ell=2$, we trivialize the representation modulo $\ell^{2n}$.)

    We replace $\mf X$, $S$, and $\sigma$ with their base-changes to $S'$, and may therefore assume that $\sigma^*\mc L$ is trivial mod $\ell^n$.
    
    Pick a closed point $s$  of $S$ (with residue field denoted by $\mb F_q$). Then the characteristic polynomial of $\sigma^*\mc L|_s$ is congruent to $(t-1)^n$ modulo $\ell^n$. 
    Note that $\sigma(s)\in \mf X_s(\mb F_q)$, and moreover that $\mc L|_{\sigma(s)}$ has Frobenius eigenvalues that are all congruent to $1\mod{\ell}$. The result now follows from Lemma~\ref{lem:rigidified}.
\end{proof}

\subsection{The case of bad reduction}

In this subsection, we study the case of local systems with bad reduction, specifically with reference to Definition \ref{def:jordan_block_mult_1}. To explain our results, we need to recall what it means for a local system to be \emph{de Rham at $\ell$}.

\begin{defn}
    Let $X/k$ be a smooth variety over a finitely generated field of characteristic 0 and let $L$ be an $\Qlbar$-local system on $X$. We say that $L$ is \emph{de Rham at all places above $\ell$} (or, more simply, \emph{de Rham at $\ell$}) if for every $\ell$-adic local field $K$ and for every embedding $k\hookrightarrow K$, the restriction $L|_{X_K}$ is de Rham at $\ell$ in the sense of \cite[Definition 2.2]{sasha}.
\end{defn}
We prove two main results: if $\mc L$ is an (geometrically irreducible) arithmetic $\Qlbar$-local system whose monodromy around a boundary component has a Jordan block of multiplicity 1 and generalized eigenvalue $\in\pm 1$, then there is a twist $\mc L(\chi)$ such that 
\begin{enumerate}

    \item the stable Frobenius trace field coincides with the Frobenius trace field; and
    \item $\mc L(\chi)$ is geometric in the sense of Fontaine-Mazur, i.e., is de Rham at all places above $\ell$.
\end{enumerate}
An example of the above boundary condition is a local system with totally degenerate unipotent reduction. These two results suggest to us that arithmetic local systems with  such conditions at the boundary are in fact strongly of geometric origin.

In particular, this shows us that cohomologically rigid local systems that have maximal unipotent degeneration admit a spreading out with bounded Frobenius traces.
\begin{setup}\label{setup:degen}
    Let $X/\C$ be a smooth variety with good compactification $\bar{X}$ and simple normal crossings boundary divisor $Z=\cup Z_i$. Let $L$ be a topological $\Qlbar$-local system with finite order determinant on $X$ with the following property: the local monodromy of $L$ around one of the boundary components has a Jordan block of multiplicity 1 and generalized eigenvalue $\pm 1$.

    Let $(\mf X, \bar{\X}, \mc L)/S$ be a spreading out, where $S=\Spec(A)$ with $A\subset \C$ and $\Z\hookrightarrow A$  smooth and $\mc L$ has fractional cyclotomic determinant.

\end{setup}

\begin{thm}

Notation as in Setup~\ref{setup:degen}. Then after potentially shrinking $S$, then there exists a twist $\mc L(\chi)$ such that the following holds.
\begin{enumerate}
    \item\label{part:bad_reduction_trace_field} The Frobenius trace field and the stable Frobenius trace field of $\mc L(\chi)$ coincide. In particular, if the stable Frobenius trace field of $\mc L(\chi)$ is bounded, then the trace field of $\mc L(\chi)$ is bounded. 
    \item\label{part:bad_reduction_dR} $\mc L(\chi)$ is de Rham at $\ell$. 
\end{enumerate}
\end{thm}
\begin{proof}
    First, shrink $S$ to ensure that the local system is tame at $\mf X_s$ for all closed points $s$. We claim both questions reduce to the case of curves. Indeed, for \ref{part:bad_reduction_trace_field}, it follows from  Lemma~\ref{lem:tame_lefschetz_trace} that question reduces a sufficiently ample curve $\bar C$ on $\bar X$ that intersects the boundary transversally. On the other hand, for \ref{part:bad_reduction_dR}, the reduction to curves is an immediate corollary of \cite{liuzhu}. 

    We therefore may assume that $\bar X$ is a curve. Let $\infty\in Z$ be a point such that $L$ has a Jordan block of multiplicity 1 and generalized eigenvalue $\lambda=\pm 1$. Then, by specialization of the tame fundamental group, by shrinking $S=\Spec(A)$ further, for every closed point $s$ of $S$,  the (geometric) local monodromy around $\infty$ of local system $\mc L_{\bar s}$ on $\mf X_{\bar s}$ has a Jordan block of multiplicity 1 and generalized eigenvalue $\lambda$.

    Let $z_{\infty}$ be a local parameter for $\infty\in \bar{\mf X}$ and write $\mc L_{\infty}$ for the resulting local system on $A((z_{\infty}))$. The $\lambda$-eigenline $\Qlbar\cdot\vec{v}\subset \mc L_{\infty}$ is preserved by $\Gal(\C((z_{\infty})))$, which means that $\piET{\Spec(A((z_{\infty})))}$ acts on $\Qlbar\cdot \vec{v}$ by a character $\psi$ of $\piET{\Spec(A)}$. Set $\chi=\psi^{-1}$.
    
    Now, by the explicit construction of Corollary~
    \ref{cor:mult_one_trace_field}, the local system $\mc L(\chi)$  satisfies (\ref{part:bad_reduction_trace_field}) of our theorem. To prove (\ref{part:bad_reduction_dR}), first we may assume that our local system $\mc L(\chi)$ has coefficients in an $\ell$-adic local field $M$. We must show that if $k=\text{Frac}(A)$, then for every $\ell$-adic local field $K$ and every embedding $k\hookrightarrow K$, the local system $\mc L_{\mf X_K}$ is de Rham at $\ell$.

    It follows from \cite{sasha} that for each choice $k\hookrightarrow K$, there exists a character $\varepsilon\colon \Gal(K)\rightarrow \Qlbar$ such that $\mc L_{\mf X_K}(\varepsilon)$ is de Rham. Our goal is to prove that the character $\varepsilon\otimes \chi^{-1}_K\colon G_K\rightarrow \Qlbar$ is de Rham. At this point, we pick an $\ell$-adic local field $M$ such that $\mc L$, $\chi$, and $\varepsilon$ all have coefficients in $M$.

    Let $R\Psi$ denote the nearby cycles functor with respect to the point $\infty$ and a choice of \'etale local coordinate around it. Then, by \cite[Corollary 4.3.4]{diao2023logarithmic}, $R\Psi$ sends de Rham local systems to de Rham local systems. Therefore $R\Psi(\mc L) \otimes \varepsilon$ is a de Rham representation of $G_K$. On the other hand, $R\Psi(\mc L)\otimes \chi$ is a representation of $G_K$ with an  invariant line. 
    Therefore, we reduce to the following Lemma \ref{lem:inv_line_dR}.
\end{proof}

\begin{lem}\label{lem:inv_line_dR}
    Let $K$ and $M$ be $\ell$-adic local fields, and let $\rho\colon G_K\rightarrow \text{GL}_n(M)$ be a continuous representation which has an invariant line. Suppose further that there exists a character $\psi\colon G_K\rightarrow M^{\times}$ such that $\rho\otimes \psi$ is de Rham. Then $\rho$ is de Rham.
\end{lem}
\begin{proof}
    It suffices to show that $\psi$ is de Rham, and since it is a character, this is equivalent to showing that it is Hodge-Tate \cite[Exercise 6.4.3]{brinonconrad}. For each embedding $\tau: M \xhookrightarrow{} \bar{K}$, let $a_{\tau}\in \bar{K}$ denote the (generalized) $\tau$-Hodge-Tate weight of $\psi$, i.e., the eigenvalue of the corresponding Sen operator--see \cite[p. 978]{sasha} for a discussion of such Hodge-Tate weights. It remains to show that $a_{\tau}\in \mb Z$ for each $\tau$. This follows since, for each $\tau$, the $\tau$-Hodge-Tate weights of $\rho \otimes \psi$ are all integers as the latter is de Rham, and zero appears as a $\tau$-Hodge-Tate weight of $\rho$ since it contains an invariant line.
\end{proof}

\begin{cor}
    Let $U$ be a smooth variety, and let $L\in\cohrig{U}{\Qlbar}$ have totally degenerate unipotent monodromy around a boundary divisor. Then $L$ admits a spreading out with bounded Frobenius traces.
\end{cor}

\subsection{A variant}
The same arguments as in the proof of Theorem~\ref{thm:main} gives the following, whose proof we omit, though we make some comments regarding it in Remark~\ref{rmk:motivic}. For simplicity, We have opted to state the version with trivial geometric determinant, i.e. $d=1$, but as usual there is a version with  determinant of order $d$.
\begin{prop}\label{prop:motivic}
Suppose $L$ is an irreducible $\Qlbar$-local system on $X$ with trivial determinant, and   of geometric origin of weight $w$, i.e. $L$ appears in cohomological degree $w$ of a smooth proper family. Then there exists a spreading out $(\X, S, \mc{L}, \iota)$. For any  spreading out such that $\mc{L}$ has weight $w$,  for all closed points $s: \Spec(k)\rightarrow S$, let $F_s$  denote the trace field of $\mc{L}|_{\X_s}$. Then 
\begin{enumerate}
\item There exists $N=N(X, L)$ such that, for all primes $m> N$ and different from the characteristic of $s$,  $F_s$ is unramified at $m$,
\item For $f=\nu !$ where $\nu$ is the number of Galois conjugates of $L$, whenever $\mb{F}_{p^f}\subset k$, $F_s$ is unramified at $p$.
\end{enumerate}
\end{prop}

\begin{rmk}\label{rmk:motivic}
    Note that part (2) of Proposition~\ref{prop:motivic} is stronger than the analogous result for (cohomologically) rigid local systems, where we must also assume that $\mb{F}_{p^{nd}}\subset k$, and in the rigid case there is no assumption on the weight of the spreading out $\mc{L}$. Recall that the proof consists of comparing $F$-isocrystals obtained in two ways: one via companions and one via Higgs-de Rham flows (or the Esnault-de Jong construction \cite[Lecture 8]{elec}).  The $\mb{F}_{p^{fdn}}$ assumption comes in because we have no control over the weights of these isocrystals.
    
    In the case of Proposition~\ref{prop:motivic}, by the assumption of being of geometric origin,  the Galois conjugates of $L$ all underlie $\mb{C}$-VHS of the same weight. Hence choosing the spreading out $\mc{L}$ to have the same weight, we have that the two $F$-isocrystals do have the same weights. 
\end{rmk}
A positive answer to the following natural seeming question implies that we can strengthen Theorem~\ref{thm:main} part (3) in the same way; note that it is certainly implied by Simpson's motivicity conjecture.
\begin{question}
Can one prove that Galois conjugates of a cohomologically rigid local system underly $\mb{C}$-VHS of the same weight?    
\end{question}
\appendix

\section{Companions}\label{section:companions}
In this section, we briefly recall the theory of ($\ell$-adic) companions, as conjectured by Deligne and proven by Lafforgue, Drinfeld, and Deligne, and record several consequences that we will use in this article. (For a panoramic summary of the theory of companions, see \cite{kedlayacompanions}.) We also explain a strong local-global compatibility of $\ell$-adic companions at the cusps, which was observed in \cite[2.2]{delignefinitude} and is a simple corollary of local-global compatibility in the Langlands correspondence.

For the rest of this section, let $Y/\mb{F}_q$ be a smooth, geometrically connected variety over a finite field of characteristic $p$.
\subsection{$\ell$-adic companions}\label{subsec:l-adic_companions}
\begin{defn}
    Let $L$ be an lisse $\Qlbar$ Weil sheaf on $Y$. We say that $L$ is \emph{algebraic} if for all closed points $y$, the characteristic polynomial of the Frobenius conjugacy class of $y$ acting on $L$,  $P_y(L,t)\in \Qlbar[t]$, has coefficients in $\bar{\Q}\subset \Qlbar$. We say that $L$ has \emph{fractional cyclotomic determinant} if there exists integers $k, m$, with $k\geq 1$, such that $\det(L)^{\otimes k}\cong \Qlbar(m)$.
\end{defn}
\begin{defn}
    Let $L$ be an algebraic lisse $\Qlbar$ Weil sheaf on $Y$. Let $\ell'\neq p$ be a prime number (possibly equal to $\ell$). Let $\tau\colon \Qlbar\rightarrow \Qlpbar$ be a (possibly non-continuous) field isomorphism; abusing notation, let $\tau$ also denote the induced isomorphism $\Qlbar[t]\rightarrow \Qlpbar[t]$. A \emph{$\tau$-companion to $L$} is a lisse $\Qlpbar$-Weil sheaf $L'$ on $Y$, such that for all closed points $y$, we have:
    $$\tau(P_y(L,t))=P_{y}(L',t)\in \Qlpbar[t].$$
\end{defn}

\begin{rmk} We make two important remarks on $\tau$-companions.
\begin{enumerate}
    
    \item As $\tau$ is generally \emph{not continuous}, we cannot simply construct a sheaf by applying $\tau$ to the corresponding matrices: the resulting representation will no longer be continuous and hence there will be no corresponding lisse Weil sheaf.
    \item Note that the behavior of the $\tau$-companion is only specified on Frobenius conjugacy classes; in particular, the definition does not directly imply any constraint for the characteristic polynomial of other element of $\piET{Y}$.
\end{enumerate}
\end{rmk}
The mathematical content of the following theorem is due to Lafforgue (for curves) and Deligne and Drinfeld (for higher dimensional varieties).
\begin{thm}\label{thm:companions}
    Let $L$ be an irreducible lisse $\Qlbar$-Weil sheaf on $Y$.
    \begin{enumerate}
        
    \item Suppose $L$ has algebraic determinant (as a rank 1 Weil sheaf). Then $L$ is algebraic and for all primes $\ell'\neq p$ and for all field isomorphisms $\tau\colon \Qlbar\rightarrow \Qlpbar$, $\tau$-companions to $L$ exist and are unique up to isomorphism.
    \item 
    Suppose $L$ has fractional cyclotomic determinant. Then $L$ extends to an \'etale sheaf on $Y$ and so do all $\tau$-companions.
    \end{enumerate}
\end{thm}
\begin{proof}
    We provide references to \cite{kedlayacompanions} for these results. The first follows from Corollary 3.10 and Theorem 4.28 of \emph{loc. cit.} The second follows from the fact that companions is preserved under tensorial operations, together with the fact that there exists a $p$-plain character $\chi$ such that $L\otimes \chi$ has finite order determinant, and is hence $p$-plain by Theorem 1.3. (For the definition of $p$-plain, see Conjecture 1.1(iii) of \emph{loc. cit.}.)
\end{proof}
\begin{notation}
    Suppose that $L$ is a semi-simple lisse $\Qlbar$ Weil sheaf on $Y$, all of whose irreducible summands have algebraic determinants. Let $\ell'\neq p$ be a prime number and let $\tau\colon \Qlbar\rightarrow \Qlpbar$ be an isomorphism. Then we denote by $^{\tau}L$ the $\tau$-companion of $L$
\end{notation}
\subsection{$p$-adic companions}
Let $\fisocd{Y}_{\Qpbar}$ denote the category of overconvergent $F$-isocrystals on $Y$ with coefficients in $\Qpbar$. As $F$-isocrystals over finite fields have a notion of ``characteristic polynomial of Frobenius'', we may simply copy the above definition of $\tau$-companions. Then, given an irreducible $\mc E\in \fisocd{Y}_{\Qpbar}$ with fractional cyclotomic determinant and any $\tau\colon \Qpbar\rightarrow \Qlbar$, the $\tau$-companion to $\mc E$ exists. For more details, see \cite{kedlayacompanions}.
\subsection{Frobenius trace field}
In this subsection, we compare the field generated by the \emph{traces} and the field generated by the \emph{coefficients} of the characteristic polynomials of Frobenius conjugacy classes at closed points for those Weil sheaves $L$ that have fractional cyclotomic determinant. To do this, we require the following well-known fact.
\begin{fact}\label{fact:brauer_nesbitt}
    Let $Y/\mb F_q$ be a smooth variety and let $L, L'$ be semi-simple \'etale $\Qlbar$-local systems on $Y$. Suppose for all closed points $y$ of $Y$, we have:

    $$\textbf{trace}(\text{Fr}_y|L)=\textbf{trace}(\text{Fr}_y|L').$$
    Then $L\cong L'$.
\end{fact}
\begin{proof}
    Let $t\colon \piET{Y}\rightarrow \Qlbar$ be the continuous function that sends $\gamma\in \piET{Y}$ to $\textbf{trace}(\gamma,L)\in \Qlbar$, and analogously for $t'$. Then $t=t'$ by the Cebotarev density theorem applied to $\piET{Y}$. It follows from the Brauer-Nesbitt theorem that $L\cong L'$ \cite[\S 20 6, VIII.376, Corollaire (a)]{bourbakialg8}.
\end{proof}

\begin{lem}
  Let $L$ be an irreducible lisse $\Qlbar$-sheaf on $Y$ of rank $n\geq 1$ with fractional cyclotomic determinant. Let $F_{\text{coeff}}\subset \Qlbar$ be the subfield generated by the coefficients of $P_y(L,t)\in \Qlbar[t]$, as $y$ ranges over the closed points of $Y$. Let $F_{\text{trace}}\subset F_{\text{coeff}}$ denote the field generated by the traces of Frobenius conjugacy classes of all closed points $y$ of $Y$ (a.k.a. the field generated by the coefficient of $t^{n-1}$ in $P_y(L,t)$). Then  $F_{\text{trace}}= F_{\text{coeff}}$.
\end{lem}
\begin{proof}
    First of all, $L$ and all of its $\ell$-adic companions extend to \'etale sheaves by Theorem~\ref{thm:companions}. Moreover, it follows from work of Deligne (building on a theorem of Lafforgue) that that $F_c\subset \Qlbar$ is a number field \cite[Theorem 1.3(i)]{kedlayacompanions}. Suppose  $F_{\text{trace}}\subsetneq F_{\text{coeff}}$. Then there exists an automorphism $\tau\colon \Qlbar\rightarrow \Qlbar$ that fixes $F_{\text{trace}}$ but acts non-trivially on $F_{\text{coeff}}$. Let $L'$ be the $\tau$-companion to $L$. Then for all closed points $y$, the trace of the Frobenius conjugacy class of $y$ on $L$ and $L'$ are the same. Fact \ref{fact:brauer_nesbitt} implies that $L\cong L'$, which means that the characteristic polynomials of all elements of $\piET{Y}$ are the same as elements of $\Qlbar[t]$, contradicting our assumption that $\tau$ was non-trivial.
\end{proof}

The above Lemma motivates the following definition.
\begin{defn}
    Let $L$ be an irreducible lisse $\Qlbar$-sheaf on $Y$ with algebraic determinant. Then the \emph{field generated by Frobenius traces}, or simply the \emph{Frobenius trace field}, is the subfield $F\subset \Qlbar$ generated by the traces of the Frobenius conjugacy classes of all closed points $y$ of $Y$.
\end{defn}
\subsection{Stable trace field}
\begin{defn}\label{def:stable_Trace}
    Let $Y/\mb F_q$ be a smooth, geometrically connected variety and let $L$ be a $\Qlbar$-local system on $Y$. The \emph{stable Frobenius trace field} of $L$ is the subfield of $\Qlbar$ given by the intersection of the Frobenius trace fields of $$L|_{Y\times_{\mb F_q}\mb F_{q'}}$$ as $q'$ ranges over the positive integral powers of $q$.
\end{defn}
The below follows directly from the definition and Fact \ref{fact:brauer_nesbitt}, and perhaps provides some intuition for the meaning of the stable Frobenius trace field.
\begin{prop}\label{prop:stable_trace_field_twist}
Let $Y/\mb F_q$ be a smooth, geometrically connected variety and let $L$ be a semi-simple $\Qlbar$-local system whose irreducible constituents have finite order determinant. Then the following are equivalent.
\begin{enumerate}
    \item The Frobenius trace field of $L$ is strictly bigger than the stable Frobenius trace field of $L$.
    \item There exists a (possibly non-continuous) $\tau\in \text{Aut}(\Qlbar)$ such that $^{\tau}L$ (the $\tau$-companion of $L$) and $L$ are not isomorphic but are isomorphic when restricted to the geometric fundamental group $\piET{\bar{Y}}$.

\end{enumerate}

\end{prop}

The following lemma shows that the stable trace field cannot be too much smaller than the trace field once we rigidify at a single point (with a sufficiently strong condition on the coefficients at that point).
\begin{lem}\label{lem:rigidified}
    Let $Y/\mb F_q$ be a smooth, geometrically connected variety and let $L$ be an irreducible $\Qlbar$-local system on $Y$ with fractional cyclotomic determinant. Let $E_1\subset \Qlbar$ be the stable trace field of $L$. Let $y$ be a $\mb{F}_q$-point of $Y$ and let $E_2\subset \Qlbar$ be a number field such that
    \begin{itemize}
        \item $P_y(L,t)\in E_2[t]\subset \Qlbar[t]$; and
        \item the roots of $P_y(L,t)$ are all congruent to $1\mod{\ell}$, and if $\ell=2$, the roots of $P_y(L,t)$ are all congruent to $1\mod{\ell^2}$.
    \end{itemize}Then the trace field of $L$ is contained in $E:=E_1\cdot E_2\subset \Qlbar$. 
\end{lem}
\begin{proof}
    Suppose the trace field of $L$ is not contained in $E$. Then there exists an automorphism $\tau\in \text{Aut}(\Qlbar)$ that fixes $E$ and a finite order character $\chi$ on $\mb F_q$ such that ${}^{\tau}L\cong L(\chi)$. (That $\chi$ is of finite order follows from the fact that companions commutes with tensorial operations and fixes integral powers of the cyclotomic character.) 

    Write the roots of $P_y(L,t)$ as $\{\lambda_1,\ldots, \lambda_n\}$. Then $\tau$ fixes this multiset, as $\tau$ fixes the coefficients of $P_y(L,t)$. On the other hand, $\tau(\{\lambda_1,\ldots, \lambda_n\})=\{\chi(y)\lambda_1,\ldots, \chi(y)\lambda_n\}$.  Therefore $1\equiv \chi(y)\lambda_1\equiv \chi(y) \mod \ell$. On the other hand,  $\chi(y)$ is a non-trivial root of unity; if $\ell$ is odd, then $\chi(y)$ is not congruent to $1\mod{\ell}$, a contradiction. (If $\ell=2$, one may repeat the above analysis, using the fact that there are no non-trivial roots of unity that are $1\mod{4}$.)
\end{proof}

\subsection{Tame Lefschetz theorem for trace fields}

\begin{lem}\label{lem:tame_lefschetz_trace}
    Let $Y/\mb F_q$ be a smooth, geometrically connected variety with simple normal crossings compactification $\bar{Y}$ and boundary divisor $D=\cup D_i$. Let $L$ be an irreducible $\Qlbar$-sheaf on $Y$ with fractional cyclotomic determinant and \emph{tame} local monodromies around the $D_i$. Let $\bar{C}\subset \bar{Y}$ be a smooth projective curve with the following properties:
    \begin{itemize}
        \item $\bar{C}$ is the smooth complete intersection of smooth ample divisors $Z_j$ in $\bar{Y}$ such that
        \item each $Z_j$ intersects the boundary divisor $D$ in good position.
    \end{itemize}
    Set $C:=Y\cap \bar{C}$. Set $F_C\subset F_Y\subset \Qlbar$ to be the Frobenius trace fields of $L|_C$, resp. $L$, as subfields of $\Qlbar$. Then $F_C=F_Y$.
\end{lem}
\begin{proof}
    First of all, the tame Lefschetz theorem implies that the map $\piTAME{C}\rightarrow\piTAME{Y}$ is surjective \cite{esnaultkindler}. If $F_C\neq F_Y$, then pick an automorphism $\tau\in \text{Aut}(\Qlbar)$ that fixes $F_C$ elementwise but does not fix $F_Y$ elementwise. Let $L'$ to be the $\tau$-companion to $L$, which is \'etale by our assumption that $L$ had fractional cyclotomic determinant. Then $L'|_C\cong L|_C$ by Fact \ref{fact:brauer_nesbitt}, which implies that $L'\cong L$, contradicting our assumption on $\tau$.
\end{proof}

\subsection{Local $\ell$-adic companions}
In this subsection, we set up a  definition of companions in the local function field setting \cite[\S 8]{deligneconstants}. We will use this explain compatibility with the local Langlands correspondence, and eventually to explain the information that the companions relations preserves at cusps. These last two points are contained in \cite[\S 2.2]{delignefinitude}. 

We first recall the definition of ``extension of scalars'' of a representation, where we explicitly include the field map $\tau$.
\begin{defn}
    Let $G$ be a group, let $\tau\colon E\rightarrow E'$ be a homomorphism between fields of characteristic 0, and let $V/E$ be a vector space. Let $\pi\colon G\rightarrow \text{GL}(V)$ be a representation. Then the representation $\tau(\pi)$ is given by the induced action of $G$ on $V\otimes_{E,\tau}E'$, i.e., by $\pi\otimes 1$. 
\end{defn}

\begin{rmk}
    We make some comments on the above definition.
    \begin{enumerate}
    
        \item Picking a basis on $V$, one may describe the representation $\tau(\rho)$ by simply applying $\tau$ to the resulting matrices.
        \item Suppose $G$ is locally compact, Hausdorff, and totally disconnected, and furthermore that $\pi$ is smooth (resp. admissible). Then $\tau(\pi)$ is also smooth (resp. admissible).
        \item Note that the notation has been chosen to distinguish $\tau$-companions and simply extending scalars by $\tau$. 
    \end{enumerate}
\end{rmk}

\begin{defn}(\cite[8.4.1]{deligneconstants})
    Let $K=\mb F_q((T))$ and let $E$ be a field of characteristic 0. A \emph{Weil-Deligne representation over $K$ with coefficients in $E$} is a triple $(\rho,V, N)$, where
    \begin{itemize}
    \item $V$ is a finite dimensional $E$-vector space;
    \item $\omega\colon W(K)\rightarrow \text{GL}(V)$ is a smooth representation of the Weil group of $\mb F_q((T))$ (i.e., an open subgroup of inertia acts trivially); and
    \item $N\colon V\rightarrow V(-1)$ is a nilpotent, $\omega$-equivariant operator.
    \end{itemize}
\end{defn}

Note that the definition of a Weil-Deligne representation is \emph{purely algebraic}, and in particular does not reflect any topology of $E$. We denote the category of Weil-Deligne representations over $K$ with coefficients in $E$ as: $\weildeligne{K}{E}$. The following proposition relates Weil-Deligne representations to \emph{continuous} representations of the Weil group; the proof amounts to a variation on Grothendieck's proof of the quasi-unipotent monodromy theorem. 
\begin{prop}
    Let $K\cong \mb F_q((T))$, let $\ell\neq p$ be prime and let $E$ be an algebraic extension of $\Ql$. Then there is an equivalence of categories:

    $$(-)_{\text{WD}}\colon \repcts{W(K)}{E}\rightarrow \weildeligne{K}{E}.$$
\end{prop}
\begin{fact}\label{fact:formula_WD}
    Given a Weil-Deligne representation $(\omega,V,N)$, the associated $E$ Weil-representation $(\rho, V)$ is given by the following formula. Let $\Phi$ be a Frobenius element in $W(K)$ and $x\in I(K)\cong \Gal(\Fqbar((T)))$; then $\rho(\Phi^ax):=\omega(\Phi^ax)\exp(-t_{\ell}(x)N)$,
    where $t_{\ell}\colon I(K)\rightarrow \Zl$ is the cyclotomic character. (Note that this formula indeed works for any algebraic extension of $\Ql$ because we are taking the exponential of a nilpotent operator; hence the associated sum is finite.)
\end{fact}

\begin{defn}
    Let $E, E'$ be fields of characteristic 0, equipped with a map $\tau\colon E\rightarrow E'$. Let $(\omega, V, N)$ be a WD representation of $K$ with coefficients in $E$. Then ${\tau}(\omega, V, N):=(\omega\otimes 1, V\otimes_{\Qlbar,\tau} \Qlpbar, N\otimes 1)$.
\end{defn}
Finally, we come to the definition of local companions.
\begin{defn}\label{def:local_companions}
    Let $\ell, \ell'\neq p$ be prime numbers (possibly the same). Let $\tau\colon \Qlbar\rightarrow \Qlpbar$ be a field isomorphism. Let $\rho\colon W(K)\rightarrow \text{GL}_n(\Qlbar)$ be a continuous representation. Then a continuous representation $\rho'\colon W(K)\rightarrow \text{GL}_n(\Qlpbar)$ is a \emph{$\tau$-companion} of $\rho$ if $${\tau}((\rho)_{\text{WD}})\cong (\rho')_{\text{WD}}$$ as WD representations with coefficients in $\Qlpbar$.
\end{defn}
    Assume that $\rho\colon W(K)\rightarrow \text{GL}_n(\Qlbar)$ is a continuous representation such that $(\rho)_{\text{WD}}$ is semi-simple. (This condition is also called Frobenius semi-simple in the literature.) Then the $\tau$-companion of $\rho$ is unique up to isomorphism, and we sometimes refer to it as $^{\tau}\rho$.

We now explain the compatibility of the local Langlands correspondence and the operation of taking $\tau$-companions.
\begin{notation}
    Let $\ell$ be a prime number and let $\rho$ be a continuous, Frobenius-semisimple representation of $W(K)$ with coefficients in $\Qlbar$. Then we denote by $\LLC{\rho}$ the associated irreducible, smooth representation of $\text{GL}_n(K)$ under the local Langlands correspondence.
\end{notation}
\begin{lem}\label{lem:loc_companions_llc}
    Let $\ell, \ell'\neq p$ be prime numbers (possibly the same). Let $\rho$ be a smooth, Frobenius-semisimple representation of $W(K)$ with coefficients in $\Qlbar$. Let $\tau\colon \Qlbar\rightarrow \Qlpbar$ be a field isomorphism. Then:
    $$\tau(\LLC{\rho})\cong \LLC{^{\tau}\rho}$$
    as irreducible, smooth representations of $\text{GL}_n(K)$ with coefficients in $\Qlpbar$.
\end{lem}

\begin{proof}
    This is stated in \cite[\S 2.2]{delignefinitude}. Alternatively, this follows quite directly from \cite[Theorem A(1)]{li-huerta_llc}. Li-Huerta's refers to smooth representations, and we may apply it to $\omega$, the underlying (smooth) representation in $(\rho)_{WD}$. Li-Huerta constructs, for every $\gamma\in W(K)$ with positive valuation\footnote{with the convention that geometric $q$-Frobenius has norm 1} and for every $h\in C^{\infty}_c(\text{GL}_n(\mc O), \Qlbar)$, a function $f_{\gamma, h}\in C^{\infty}_c(\text{GL}_n(K))$, with the following property: a \emph{smooth} Frobenius-semisimple Weil representation $\omega$  of rank $n$ with coefficients in $\Qlbar$ and an irreducible smooth representation $\pi$ of $\text{GL}_n(K)$ with coefficients in $\Qlbar$ correspond under the local Langlands correspondence if and only if:

    $$\textbf{trace}(f_{\gamma, h}|\pi)=\textbf{trace}(\gamma|\omega)\textbf{trace}(h|\pi).$$
    Then,  $\textbf{trace}(\gamma|\tau(\omega))=\tau(\textbf{trace}(\gamma|\rho))\in \Qlpbar$. 
    \end{proof}
    
\subsection{Compatibility at the cusps}

In this subsection, we explain how the companions relation interacts with ``local monodromy at the cusps'' \cite[\S 2.2]{delignefinitude}. We then show that, when the local monodromy at a cusp has a particular form (e.g. is totally degenerate), then after twisting, the stable trace field is ``big'', i.e., is the trace field.
 The following theorem is essentially an observation of Deligne. Its novelty lies in the fact that, in the definition of companions, only the behavior at closed points is determined; however, the behavior at the cusps will be similarly forced, by local-global compatibilty in the Langlands correspondence.

Let $C/\mb F_q$ be a smooth, geometrically connected curve with smooth compactification $\bar{C}$. If $L$ is a local system on $C$ and $\infty \in \bar{C}\setminus C$ is a closed point with residu field $\mb F_{q'}$ and local uniformizer $z_{\infty}$, then we denote by $L_{\infty}$ the restriction of $L$ to $\Spec(\mb F_{q'}((z_{\infty})))$.
\begin{thm}\label{thm:cusps_compatible}
    Let $L$ be an irreducible $\Qlbar$-local system on $C$ with fractional cyclotomic determinant, let $\infty\in \bar{C}\setminus C$ be a closed point, with local parameter $z_{\infty}$ and residue field $\mb F_{q'}$. Let $E\subset \Qlbar$ be the Frobenius trace field of $L$. Then the following hold.
\begin{enumerate}
    \item\label{part:phi_alg} If $\Phi\in \Gal(\mb F_{q'}((z_{\infty})))$ is a lift of the $q'$-Frobenius,  the  characteristic polynomial of $\Phi$ on $L_{\infty}$ has coefficients in $E\subset \Qlbar$; in particular, the eigenvalues of $\Phi$ are algebraic numbers.
    \item\label{part:cusps_compatible} Let $\tau\in \text{Aut}(\Qlbar)$. Then, after potentially Frobenius-semisimplifying, the local representations $L_{\infty}$ and $(^{\tau}L)_{\infty}$ are $\tau$-companions in the sense of Definition~\ref{def:local_companions}.
\end{enumerate}    

\end{thm}
\begin{proof}
This is also essentially contained in \cite[\S 2.2]{delignefinitude}; we simply reprise the proof. We first prove \ref{part:cusps_compatible}. By twisting by a fractional cyclotomic character, we may assume that $L$ has finite order determinant. Set $\Pi$ to be the global irreducible cuspidal $\Qlbar$ automorphic representation, attached to $L$. Then $\tau(\Pi)$ is a global irreducible cuspidal $\Qlpbar$ automorphic representation.\footnote{Inasmuch as an automorphic representation is a sub-vector space of a space of functions from a discrete space to $\Qlbar$, one can simply post-compose the functions with $\tau$ to obtain functions with values in $\Qlpbar$.} Let $\Pi_{\infty}$ denote the component of $\Pi$ at $\infty$ (and analogously for $(\tau(\Pi))_{\infty}$).
Then, again by the Langlands correspondence, there is an associated $\Qlpbar$-local system $^{\tau}L$.\footnote{local-global compatibility for closed points $c$ of $C$ implies that $\tau(P_c(L,t))=P_c(^{\tau}L,t)\in \Qlpbar[t]$.} Local-global compatibility at $\infty$ \cite[Corollaire VII.5]{lafforgue} implies that $\LLC{L_{\infty}}\cong \Pi_{\infty}$, where the latter denotes the component at $\infty$ of $\Pi$ and also that $\LLC{(^{\tau}L)_{\infty}}\cong \tau(\Pi)_{\infty}\cong \tau(\Pi_{\infty})$. Then Lemma~\ref{lem:loc_companions_llc} implies that $(^{\tau}L)_{\infty}$ and $ ^{\tau}(L_{\infty})$ are isomorphic after Frobenius-semisimplifying, as desired. 

We now prove \ref{part:phi_alg}. It suffices to show that for any $\tau\in \text{Aut}(\Qlbar)$ that fixes $E$, we have: $^{\tau}(L_{\infty})\cong L_{\infty}$. This follows  from \ref{part:cusps_compatible} and the fact that $^{\tau}L\cong L$.
\end{proof}
\begin{rmk}
    Although there are uncountably many choices of Frobenius lifts, it turns out that the eigenvalues of a Frobenius lift acting on $L_{\infty}$ is independent of the choice of lift. This essentially follows from Grothendieck's quasi-unipotent monodromy theorem and is implicit in the construction of the associated Weil-Deligne representation. The key fact about matrices is the following: if $K$ is a field of characteristic $0$, $M,U\in \text{GL}_n(K)$, with $U$ unipotent, such that $MUM^{-1}=U^q$, then the spectrum of $M$ and the spectrum of $UM$ are the same. We sketch a proof of this fact. We may assume that $U\neq 1$. First of all, $M=U^qMU^{-1}=U (U^{q-1}M)U^{-1}$. Plugging in the formula for $M$ into the right hand side, one deduces further that $M=U^{q}U^{q}MU^{-1}U^{-1}=U^{2q}MU^{-2}=U^{2}(U^{2q-2}M)U^{-2}$. It therefore follows by induction that for each $r\geq 1$, $M$ is conjugate to $U^{r(q-1)}M$.

    Consider the set $\{U^rM\}_{r\in \mathbb Z}\subset \text{GL}_n(K)$; set $S$ to be its Zariski closure. Then $S$ is one-dimensional (because the Zariski closure of $\{U^r\}$ is one-dimensional). There are infinitely many points of $S$ with the same spectrum as $M$. As the set of points with given spectrum inside of $\text{GL}_n(K)$ is a Zariski closed subset, it follows that all of $S$ has the same spectrum, namely, that of $M$.
\end{rmk}

We use this compatibility to show that in certain cases, the trace field is identically the stable trace field.
\begin{defn}\label{def:jordan_block_mult_1}
    Let $E$ be an algebraically closed field of characteristic 0 and let $A\in \text{Mat}_{n\times n}(E)$ be a matrix. We say that $A$ has a \emph{Jordan block of multiplicity 1 and generalized eigenvalue $\lambda$} if, in the Jordan decomposition of $A$,
    $$A=\bigoplus_{\mu\in E, s\in \Z_{\geq 0}}J(\lambda, r)^{m(\mu, s)},$$
    there $r\geq 1$ such that $m(\lambda,r)=1$. In this Jordan block, there is a unique invariant line, which we call the \emph{$\lambda$-eigenline.}
\end{defn}
\begin{cor}\label{cor:mult_one_trace_field}
    Let $L$ be an irreducible $\Qlbar$-local system on $C$ with fractional cyclotomic determinant, let $\infty\in \bar{C}\setminus C$ be a closed, $\mb F_q$-rational point, with local parameter $z_{\infty}$. Suppose that $L_{\infty}$ is tamely ramified, and let $\eta$ denote a generator of tame inertia of $\Gal(\mb F_{q}((z_{\infty})))$.  Suppose that the action of $\eta$ on $L_{\infty}$, which is automatically quasi-unipotent, there exists $r$ and $\lambda=\pm1$ such that $J(\lambda, r)$ appears with multiplicity one, i.e.  $m(\lambda, r)=1$.
    
     If a local Frobenius lift $\Phi\in \Gal(\mb F_{q}((z_{\infty})))$ acts trivially on the  $\lambda$-eigenline of $\eta$, then the stable Frobenius trace field of $L$ is equal to the Frobenius trace field of $L$.
      
\end{cor}
\begin{proof}
 Let $\chi\colon \Gal(\mb F_q)\rightarrow \Qlbar^{\times}$ be a character. Then on $L(\chi)_{\infty}$, the element $\Phi$ acts on the $\lambda$-eigenline by $ \chi(\Phi)$. Suppose there exists a character $\chi$ and a field automorphism $\tau\in \text{Aut}(\Qlbar)$ such that  $L(\chi)$ is a $\tau$-companion of $L$; it suffices to prove that $\chi$ is trivial, so let us assume, for the sake of contradiction, that $\chi(\Phi)\neq 1$.  By Theorem~\ref{thm:cusps_compatible}, we have $^{\tau}L_{\infty}\cong\  (L_{\infty}(\chi))$. Note further that $^{\tau}L_{\infty}$ is tame, as tameness is preserved by the companions relation.

Note that, by the relation  $\Phi \eta \Phi^{-1} = \eta^q$,  the unique copy of $J(\lambda,r)$ inside of $L_{\infty}$ is in fact a $\text{Gal}(\mb F_q((z_{\infty})))$-subrepresentation; we call this sub representation $L_{\infty}(\lambda,r)$. On the one hand, by \ref{lem:loc_companions_llc} it follows that on $^{\tau}L_{\infty}$, a tame representation of $\Gal(\mb F_q((z_{\infty})))$, our fixed generator $\eta$ of tame inertia has a Jordan black of multiplicity 1 and generalized eigenvalue $\tau(\lambda)=\lambda$; moreover, $\Phi$ acts trivially on  the $\lambda$-eigenline in the Jordan block $J(\lambda, r)$ of $^{\tau}L_{\infty}$. In the above language,  the $\Gal(\mb F_q((z_{\infty})))$-representation $^{\tau}L_{\infty}(\lambda, r)$ has a unique rank 1 subrepresentation, and moreover on this subrepresentation $\Phi$ acts trivially.

On the other hand, in $L_{\infty}(\lambda, r)(\chi)$, there is a unique rank one $\Gal(\mb F_q((z_{\infty})))$-subrepresentation; on this line,  $\Phi$  acts by $\chi(\Phi)\neq 1$. This contradicts the fact that $^{\tau}L_{\infty}(\tau(\lambda), r)\cong L_{\infty}(\lambda, r)(\chi)$, as desired.

\end{proof}

\printbibliography[]

\end{document}